\documentclass[english]{article}
\usepackage[T1]{fontenc}
\usepackage[latin9]{inputenc}
\usepackage{amsmath}
\usepackage{amssymb}
\usepackage{babel}
\usepackage{amsthm}
\usepackage{graphicx}
\usepackage{tikz}
\usepackage[all]{xy}
\usepackage{bm}
\usepackage{comment}
\usetikzlibrary{calc,patterns,angles,quotes}
\usepackage[toc,page]{appendix}
\usepackage{hyperref}
\usepackage{authblk}

\newtheorem{prop}{Proposition}
\newtheorem{theorem}{Theorem}
\newtheorem{lemma}{Lemma}
\newtheorem{conjecture}{Conjecture}
\newtheorem{rem}{Remark}
\newtheorem{cor}{Corollary}
\newtheorem{defi}{Definition}

\begin{document}

	\AtEndDocument{%
		\par
		\medskip
		\begin{tabular}{@{}l@{}}%
			\textsc{Karlsruhe Institute of Technology, Karlsruhe, Germany.}\\
			\textit{E-mail address}: \texttt{airi.takeuchi@partner.kit.edu}
	\end{tabular}
\vspace{1cm}

\begin{tabular}{@{}l@{}}%
	\textsc{University of Augsburg, Augsburg, Germany.}\\
	\textit{E-mail address}: \texttt{lei.zhao@math.uni-augsburg.de}
\end{tabular}}

\title{Conformal Transformations and Integrable Mechanical Billiards}
\author{Airi Takeuchi, Lei Zhao}



\date{}
\maketitle

\begin{abstract}
{In this article we explain that several integrable mechanical billiards in the plane are connected via conformal transformations. We first remark that the free billiard in the plane are conformal equivalent to infinitely many billiard systems defined in central force problems on a particular fixed energy level. We then explain that the classical Hooke-Kepler correspondence can be carried over to a correspondence between integrable Hooke-Kepler billiards. As part of the conclusion we show that any focused conic section gives rise to integrable Kepler billiards, which brings generalizations to a previous work of Gallavotti-Jauslin \cite{Gallavotti-Jauslin}. {We discuss several generalizations of integrable Stark billiards. }We also show that any confocal conic sections give rise to integrable billiard systems of Euler's two-center problems.}
\end{abstract}

\tableofcontents

\begin{quote}
		\subsection{{General Setting of Mechanical Billiard Systems}}
        The dynamics of billiards in the plane in which a particle moves freely along straight lines in a "billiard table" and reflects elastically at a reflection wall is a widely-studied subject. 
        In this paper we study a type of variants of such systems, namely planar mechanical billiards, in which the particle is assumed to move under the additional influence of a conservative force field derived from a {potential}. 
	
	Our general setting is the following: We consider a mechanical system on two-dimensional Riemanian manifold $(M, g)$ with a {force function} $U: M \to \mathbb{R}$. {The potential is  $V=-U$.}
	The dynamics is given by the corresponding second-order Newton's equation
	$$\nabla_{\dot{q}}\dot{q}=  \nabla_{g} U(q), \, q \in M,$$
	in which $\nabla$ is the Levi-Civita connection of $g$.
	Moreover, we assume that the motion is elastically reflected against a $C^1$-smooth curve $\mathcal{B}\subset M$. This then defines a billiard system when we specify (when necessary) a component of $M \setminus \mathcal{B}$ as a billiard table where the motions of interest take place. We shall not need this specification for the purpose of this article. We thus define the corresponding mechanical billiard system as the quadruplet $(M,g,U,\mathcal{B})$.

	
	Note that as compared to the case of free billiards, it is not always necessary to assume that the billiard table is bounded in order for the billiard mapping to be well-defined, {for example when the force forces the trajectories to meet the reflection wall $\mathcal{B}$ again}. Moreover, in such cases we may as well remove part of $\mathcal{B}$ which may possibly lead to a still well-defined, albeit discontinuous billiard mapping.

	A first integral of the system $(M,g,U,\mathcal{B})$ is a first integral of $(M, g, U)$ which is invariant under the reflections at $\mathcal{B}$. The energy ${E = T {-} U}$ is always a first integral of the system. As we are in dimension 2, such a system is called integrable if there exists another first-integral of this system independent of $E$. 
	
	Due to the conservation of energy $E$, we can moreover restricted the mechanical billiard system to an energy hypersurface $\{E = e\}$. We denoted the corresponding billiard system by $(M,g, U, \mathcal{B}, e)$. Accordingly, this restricted system is called integrable if there exists an additional non-trivial first integral of the system defined on $\{E=e\}$. In this article, we primary use this definition of integrability since it is natural to fix its energy when we consider a billiard system. 
	
	The free motion case ($U = 0$) corresponds to the classical free billiards.  In this case, any of its positive energy hypersurfaces carry the same dynamics. In contrast to this, a general mechanical systems can have essentially different behaviors on different energy surfaces and analogously also the mechanical billiard systems. Therefore it is often necessary to specify the energy values $e$ or the subset of possible energy values $\mathcal{E}$ under consideration. We write $(M,g,U, \mathcal{B}, \mathcal{E})$ to emphasize also the region of energy under consideration. Such a system is integrable if the system is integrable for all $e \in \mathcal{E}$. On the other hand, a ``reflection wall'' $\mathcal{B}$ such that $(M,g,U, \mathcal{B}, \mathcal{E})$ is integrable, is called an integrable reflection wall for the mechanical system $(M,g,U, \mathcal{B}, \mathcal{E})$. Note that for the discussion of integrability, we do not require that the billiard mapping to be always well-defined. 

	Already in the free billiard case with no additional force, billiard systems may carry rich dynamics and offers class of examples illustrating many dynamical phenomena \cite{Tabachinkov}. The book \cite{Kozlov-Treshchev} also discusses several aspects of mechanical billiards. 
	
	\subsection{ {Known Examples of Integrable Mechanical Billiard Systems}}
	{For free motion in 2-dimensional plane $\mathbb{R}^2$, there are two types of integrable billiard systems. The simplest one is the one with a circular reflection wall. In this case, one can easily see that the angle of reflection is preserved, hence it is an additional first integral. The second example is provided with an elliptic reflection wall. The integrability of such a system has been shown by Birkhoff \cite{Birkhoff}. This integrability can be generalized in the case of free motions in 2-dimensional sphere $\mathbb{S}^{2}$ and the hyperbolic space $\mathbb{H}^2$, in which circular and elliptic reflection walls are also integrable \cite{Veselov-Alexander}\cite{Tabachnikov3}. Additionally, a conjecture attributed to Birkhoff and Poritsky states that any closed convex reflection wall of an integrable billiard system is either a circle or an ellipse \cite{Poritsky}. This conjecture has not been fully proven yet, but there are important progresses recently made \cite{Kaloshin}. Also, an algebraic version of the conjecture for billiards on the plane and constant curvature surfaces has recently been proved by Glutsyuk \cite{Glutsyuk_2017}\cite{Glutsyuk_2020}. 	}
	
	Many examples of integrable mechanical billiard systems with the presence of a non-constant {potential} function have been identified as well.  We start our list with a class of relatively easy examples: In a central force problem in $\mathbb{R}^{2}$, in which $V$ is a function of $|q|$ only, then circles with center at $O$ and lines passing through the center {$O$} are integrable reflection walls: In both cases, it is direct to check that the norm of the angular momentum is preserved under reflections at these reflection walls. The very same argument works also on the sphere $\mathbb{S}^{2}$, and on the hyperbolic plane $\mathbb{H}^{2}$.  
	
	 {A number of integrable mechanical billiards are defined for the Kepler problem and the Hooke problem, with respectively force functions of the forms $U = \frac{s}{r}$ and $U=f r^{2}$, where $r$ is the distance of the particle from a fixed center $O \in \mathbb{R}^{2}$ and the factors $f, s \in \mathbb{R}$ can take both signs, allowing both attractive and repulsive forces. }
	 
	  {In the Hooke problem}, it is direct to see that any line is integrable. Centered conic sections are also integrable, for which the case of an centered ellipse follows from the classical work of Jacobi on the integrability of a quadratic radial potential of the form $r^{2}$ restricted to a triaxis ellipsoid in $\mathbb{R}^{3}$, by letting one of the axis of the ellipsoid tends to zero \cite{Jacobi Vorlesung}{\cite{Fedorov}}. Additionally, the integrability of two centered confocal elliptic reflection walls is shown by Pustovoitov in \cite{Pustovoitov2019}. Later, by the same author, the integrability of reflection walls consist of centered confocal ellipses and centered confocal hyperbola is also established \cite{Pustovoitov2021}. {In addition, the centered elliptic reflection walls are integrable for certain potentials given by certain polynomials of even degrees in $\mathbb{R}^2$ \cite{Kozlov-Treshchev}\cite{Wojciechowski}.}
	
	The Kepler problem in $\mathbb{R}^2$ with a line not passing through the attractive center is contained in a class of mechanical billiard systems proposed by Boltzmann in \cite{Boltzmann}, who expected that such systems {to be} ergodic and in particular non-integrable. Based on a close examination of Boltzmann's argument and some numerical investigations, Gallavotti has conjectured that the contrary is actually true, namely this mechanical billiard system {should actually be integrable.} This has been confirmed by Gallavotti and Jauslin in \cite{Gallavotti-Jauslin}, with alternative proofs in \cite{Felder} and \cite{Zhao}. Moreover, such systems can be generalized to $\mathbb{S}^{2}$ and  $\mathbb{H}^{2}$ \cite{Zhao}.
	
	It has been also known that a parabolic reflection wall whose focus is at
	the origin is integrable for the {Stark} problem in which the {potential is a linear combination
	of a Kepler and a uniform gravitational potential $U = g y$ with constant $g \in \mathbb{R}$} \cite{Korsch-Lang}. This result has its significance in optics, and such a parabolic mirror has been constructed in experiments \cite{Feldt-Olafsen}.
	In Section \ref{sec: duality_Hook_Kepler}, we shall give a short alternative proof of the theorem of  \cite{Korsch-Lang} as well as bring certain extensions.

	{More recently, for the planar system with potential $U := \dfrac{k}{2}(x^2 + y^2) + \dfrac{\alpha^2}{2 x^2} + \dfrac{\beta^2}{2 y^2}$, Kobtsev showed that any centered ellipse with semi-axis $a,b$ forms an integrable reflection wall \cite{Kobtsev}.}
	
	{The integrable dynamics of some of these integrable mechanical billiards have been extensively investigated as well. For this we refer to \cite{Fomenko-Vedyushkina} and the references therein.}
	
	\subsection{{Purpose of this Article}}
These examples of integrable mechanical billiards have been found independently under different contexts. In this {article, our main goal is to illustrate how conformal transformations transform integrable mechanical billiard systems. }

{As application}, we shall start by showing that via conformal transformations one gets from integrable free billiards in the plane some classes of planar immersed curves which are integrable reflection walls for certain central force problem in the plane on its zero-energy level. The complexity of these curves makes us wonder whether this simple corollary admit different but as simple solutions, if we first fix the potential and ask to identify these integrable reflection walls.

We shall then apply the well-known complex square mapping,
{\cite{Maclaurin}\cite{Goursat}\cite{Levi-Civita}\cite{Levi-Civita_1904}} 
which induces a duality between the Hooke and the Kepler problems, to obtain a duality between integrable Hooke and Kepler billiards. {The same transformation also leads to many new classes of integrable mechanical billiards in systems similar to the Stark problem.} We shall also apply a closely-related conformal mapping due to Birkhoff to Euler's two-center problem and identify its integrable reflection walls.

{In this way many known examples of integrable mechanical billiards are related.} Besides, we have also identified some classes of integrable billiards {which we think are new}, namely 
\begin{itemize}
\item {conic sections focused at the center } for the Kepler billiards;
\item {well-oriented parabola focused at the center for Stark-type billiards};
\item confocal conic sections for Euler's two-center problem.
\item {Moreover, some of these integrable conic section reflection walls in the Kepler and in the two-center problem are allowed to be combined when they are confocal.}
\end{itemize}

	We organize this article as follows:
	
	{In} Section \ref{sec: conformal_trans}, {we} introduce conformal transformations between mechanical billiard systems. In particular, we explain that conformal transformations preserve integrability of mechanical billiards. {As a first application,} we show that with conformal transformations we get infinitely many families of planar mechanical billiards which are integrable at one particular energy level.
	
	In Section \ref{sec: duality_Hook_Kepler}, we explain the duality between the Hooke billiard and the Kepler billiard and establish our results concerning them.  
	
	{In Section \ref{sec: Stark}, we study the integrability of Stark-type mechanical billiards. In particular, we provide a short alternative proof to the theorem of Korsch-Lang \cite{Korsch-Lang}.}
	
	In Section \ref{sec: two_center_problem}, we apply {Birkhoff's conformal transformation} 
	to the classical Euler's two-center problem and establish our results concerning this system. 
	
	
	\section{{Conformal Transformations and Mechanical {Billiards}}}
	\label{sec: conformal_trans}
	\subsection{{Duality between Integrable Mechanical Billiards}}
	{
	We start our discussion by the following definition of integrable mechanical system.
	\begin{defi}
		Let $(M, g)$ be a 2-dimensional Riemannian manifold, $U$ a smooth function on $M$, $\mathcal{B} \subset M$ a $C^{1}$-curve, and $\mathcal{E} \subset \mathbb{R}$ such that $(M, g, U, \mathcal{B}, \mathcal{E})$ is a 2-dimensional mechanical billiard, meaning that $(M, g, U)$ is a natural mechanical system and the motions are assumed to carry energies from $\mathcal{E}$ and are reflected elastically at $\mathcal{B}$.  We call the system $(M, g, U, \mathcal{B}, \mathcal{E})$ \emph{integrable} when there exists an additional $C^{\infty}$ function 
		\[
		G: T^*M \to \mathbb{R} 
		\] 
		 independent of its energy $E$, which is preserved by the motions and by reflections at $\mathcal{B}$.
	\end{defi}	
	\begin{defi}
		Let $M $ and $M'$ be two smooth manifolds and $\phi:M \to M'$ be a {$k$-to-$1$ regular mapping}.  
		Then its cotangent lift $\Phi : T^*M \to T^*M'$ is defined as
		\[
		\Phi(x, \xi) = (x', \xi '), \quad x \in M,~\xi \in T^*_xM, ~x' \in M', ~\xi' \in T^*_{x'}M',
		\]
		with
		\[
		x' = \phi(x), \quad \xi' = (d\phi_x^*)^{-1} \xi,
		\]
		where $(d\phi_x^*)^{-1}$ is the inverse mapping of the isomorphism $d\phi_x^* : T^*_{\phi(x)}M' \to T^*_{x}M$ that is an adjoint of the derivative $d \phi_x : T_x M \to T_{\phi(x)}M$ at $x$. 
	\end{defi} 
	
	{Moreover, $\Phi$ preserves the canonical symplectic forms on the contangent bundles.}	
	More precisely we shall show that the cotangent lift $\Phi$ pulls the tautological one-form $\alpha$ on $T^*M'$ back to the tautological one-form $\alpha'$ on $T^*M$, i.e. $\Phi^* \alpha'= \alpha$. This means pointwise
	\[
	(d \Phi)^*_p (\alpha')_{p'} = (\alpha)_p,
	\]
	where $(d \Phi)^*_p$ is {the} adjoint of the derivative $d \Phi$ at $p$ and  $p'=\phi (p)$.
	Let $\pi: T^*M \to M$ and $\pi': T^*M' \to M'$ be {footprint} projections such that 
	\[
	\pi(x, \xi )= x, \quad \pi'(x', \xi')= x', \quad x \in M, \xi \in T^*M, x' \in M', \xi' \in T^*M'. 
	\]
	The tautological one-forms $\alpha,\alpha'$ are defined pointwise as
	\[
	(\alpha)_p = (d \pi)_p^* \xi, \quad (\alpha')_p = (d \pi')_{p'}^* \xi', 
	\]
	where $p = (x, \xi), p'= (x', \xi')$ and $ (d \pi)_p^*$, $(d \pi')_{p'}^*$ are {adjoints} of the derivatives of $\pi$ and $\pi'$ at $p$ and $p'$ respectively.
	We now have
	\begin{align*}
	(d \Phi)^*_p (\alpha')_{p'} &= (d \Phi)^*_p (d \pi')_{p'}^* \xi' = (d (\pi' \circ  \Phi))^*_p  \xi' = (d ({\phi} \circ \pi))^*_p  \xi'\\
	&= (d \pi)^*_p (d {\phi})^*_p \xi' = (d \pi)^*_p \xi= (\alpha)_p.
	\end{align*}
}
	
	Now we are ready to state our first theorem. 
	\begin{theorem}
		\label{thm: conformal_trans}
		Let $(M,g,U, \mathcal{B}, \mathcal{E})$ and $(M',g',U', \mathcal{B}', \mathcal{E}')$ be two 2-dimensional natural mechanical systems, where $\mathcal{E}$ and $\mathcal{E'}$ consist of regular values of the energies. { Let  $\phi:M \to M'$ be a conformal $k$-to-$1$ smooth regular mapping for some $k \in \mathbb{N}_{+}$} 
		and assume that $ \phi(\mathcal{B}) \subset \mathcal{B'} $. Suppose also that its cotangent lift $\Phi : T^*M \to T^*M'$ sends each energy  {hypersurface} with energy $ e \in \mathcal{E}$ to an energy {hypersurface} with energy in $ e' \in \mathcal{E}'$.
		
		{Under these assumptions, if $(M',g',U', \mathcal{B}', \mathcal{E}')$ is integrable, then $(M,g,U, \mathcal{B}, \mathcal{E})$ is also integrable. 
		Additionally, if $(M,g,U, \mathcal{B}, \mathcal{E})$ is integrable and $\psi (\mathcal{B}') \subset \mathcal{B}$ for a 
		{smooth} inverse branch $\psi: M' \to M $ of $\phi: M \to M'$,
		and  its cotangent lift $\Psi : T^*M' \to T^*M$ sends each energy {hypersurface} with energy $ e' \in \mathcal{E}'$ to an energy {hypersurface} with energy $e \in \mathcal{E}$,
		then $(M',g',U', \mathcal{B}', \mathcal{E}')$ is also integrable.}\
	\end{theorem}
	\begin{proof}
		We first suppose that $(M',g',U', \mathcal{B}', \mathcal{E}')$ is integrable.
		Since the {energies} from $\mathcal{E'}$ and $\mathcal{E}$ are mapped to each other, the vector fields $X_{H}$ and $X_{\Phi^* H'}$ leave the common energy hypersurface $$\{H = e\}=\{\Phi^* H'=e'\} \quad e \in \mathcal{E}, e' \in \mathcal{E'}$$
		invariant, on which both {vector fields} are non-vanishing by the assumption that $e$ is a regular value of $H$. 
		Thus, there exists a smooth function $\rho : T^* M \to \mathbb{R} \backslash \{ 0 \}$ such that $X_{H} = \rho X_{\Phi^* H'}$. This means $X_{H}$ and $X_{\Phi^* H'}$ {agree} up to time parametrization.
		
		From integrability of $(M',g',U', \mathcal{B}',e')$, there exists {a} first integral $G'$ that is independent of energy $H'$. {Thus}
		\begin{equation*}
		\mathcal{L}_{X_{H'}} G '|_{H' = e'} =  \{ H', G'  \} |_{H' = e'} = 0,
		\end{equation*}
		where $\mathcal{L}_{X_{H'}}$ is the Lie derivative along the vector field $X_{H'}$.
		By setting $G :=  \Phi^ * G'$, we obtain
		\begin{equation*}
		\mathcal{L}_{X_H} G |_{H = e} =  \rho	\mathcal{L}_{X_{\Phi^* H'}} G |_{H = e}= \rho\{\Phi^* H', G  \}|_{H = e} = 0.
		\end{equation*}
		So, $G$ is conserved along the flow {of $X_H$ on $\{H = e\}$.}
		
		Now, we check the conservation of $G$ before and after the reflection at $\mathcal{B}$. Take a point $b \in \mathcal{B}$, then $b' = \phi(b)$ lies in $\mathcal{B'}$ from the assumption $\phi(\mathcal{B}) \subset \mathcal{B'} $. Let $(v'_{-}, v'_{+})$ be a pair of incoming and outgoing vector at $b' \in \mathcal{B'}$ so that $v'_{-}$ and $v'_{+}$ have the same $g'$-metric and angles they made with the normal agree up to sign. There exists $(v_{-}, v_{+})$ such that $(d\phi_b( v_{-}), d\phi_b (v_{+})) = (v'_{-}, v'_{+})$. From the conformality of $\phi$, the vectors $v_{-}$ and $v_{+}$ have the same $g$-metric and the angle with the normal agree at $b \in \mathcal{B}$ up to sign. Therefore $(v_{-}, v_{+})$ are vectors before and after an elastic reflection at $b$. Since $G'$ is invariant under the reflection at $\mathcal{B'}$, $G:=  \Phi^ * G'$ is then invariant under the reflection at $\mathcal{B}$. 
		
		{We now suppose that $(M,g,U, \mathcal{B}, \mathcal{E})$ is integrable. Since $\phi: M \to M'$ is a regular $k$-to-$1$ covering map, there exists $k$ smooth regular inverse branches of $\phi$. 
		Let $\psi: M' \to M$ be such an inverse branch. The above argument now works the same for $\psi$ in place of $\phi$}.
	
\end{proof}
	
	\begin{rem} When the billiard mappings are well-defined, then the above theorem actually shows that they are (semi-)conjugate which implies their equivalence {(up to covering)} in the sense of dynamical systems.  
	\end{rem}
	
	\begin{rem}  The theorem can be directly generalized to certain multi-dimensional case as well. Nevertheless, in view of Liouville's theorem, conformal mappings on {a domain} of $\mathbb{R}^{d}, d \ge 3$ are rather limited. Thus we may in general expect more non-trivial applications in the two-dimensional case.
	\end{rem}
	
	Pulling back a mechanical billiard system by the cotangent lift of a conformal mapping without restricting its energy gives another mechanical billiard system without the necessity to change time, where the kinetic energy is transformed into a quadratic form of velocity depending on the base point in the configuration space. In our applications, we shall rather fix its energy and make proper time change in order to have an iso-energetic correspondence between mechanical billiards in the plane with standard kinetic energies. 
	
	{Before applying this theorem to concrete problems, we first state a lemma concerning the time-reparametrization of a Hamiltonian system on a fixed energy hypersurface.}
	\begin{lemma}
		\label{lem: time-reparametrization}
		{Let $H(p, q)$ be a Hamiltonian function defined on $T^{*} \mathbb{R}^{2}$ equipped with its canonical symplectic form. }Set $\hat{H}:=g(q) \cdot H$ where $g(q)>0$ is a  $C^{\infty}$-smooth function of $q$. Then the two systems defined by $H$ and $\hat{H}$ are equivalent up to a time-reparametrization given by {$d \hat{t}= d t/g(q)$} on {their} zero {energy-hypersurfaces.}
	\end{lemma}
	\begin{proof}
		
		The statement immediately follows from the equation of motion:
		\begin{align*}
		&\dot{q}=\frac{\partial \hat{H}}{\partial p} = g(q) \cdot \frac{\partial H}{\partial p},\\
		&\dot{p}=\frac{\partial \hat{H}}{\partial p} = g'(q) \cdot  H +  g(q) \cdot \frac{\partial H}{\partial q}=g(q) \cdot \frac{\partial H}{\partial q},
		\end{align*}
	{when being restricted to their common zero energy-hypersurfaces.}
	\end{proof}
	
	{We now apply Theorem \ref{thm: conformal_trans} to some central force problems.}
	\begin{theorem}
		\label{thm: duality}
		{Let $f, s  {\in \mathbb{R}}$ be two real parameters.} For any $k \in \mathbb{N}, k \geq 2$ the cotangent lift of the conformal mapping {$$\mathbb{C} \setminus O \mapsto  \mathbb{C} \setminus O, z \mapsto q = z^k$$} gives the transformation between two hamiltonians 
		\begin{equation*}
		\frac{|w|^2}{2} + f |z|^{2k -2} + s
		\end{equation*}
		and
		\begin{equation*}
		\frac{|p|^2}{2} + \frac{s}{|q|^{2- 2/k}} + f
		\end{equation*}
		on {their} zero-energy surface, up to time parametrization. 
		
		In particular, it gives
		\begin{itemize}
			\item for {$(f >0, s<0)$, $(f<0, s>0)$, or $(f<0, s<0)$}, an iso-energetic transformation between two central force systems.  {In particular, when $k=2$ and $f >0, s<0$, between the Hooke system of isotropic harmonic oscillators and the Kepler system in the plane.}
			\item for $(f=0, s<0)$ or $(s=0, f<0)$, an iso-energetic transformation between the free motion in the plane with positive energy and some homogeneous central force systems at their energy zero.
			\item for $f = s = 0$, a trivial iso-energetic transformation between zero-energy free motions in the plane.
		\end{itemize}
	\end{theorem}
	\begin{proof}
		
		{The cotangent lift for $\phi: z \mapsto z^k$ is given by 
		\[
		\Phi: (z,w) \mapsto \left(q=z^k, p=\frac{w}{k \overline{z}^{k-1}}\right).
		\]
		and is a symplectic map. This follows from our discussions above but it is also {direct} to have a verification with complex notations. Indeed, the canonical symplectic form $\omega_0 = \sum d q_i \wedge d p_i$ is given by $\omega_0= d \alpha_0$, where $\alpha_0= \sum p_i d q_i$ is the tautological one-form. When we identify $\mathbb{R}^2$ and $\mathbb{C}$ and describe $p = p_1 + i p_2$ and $q = q_1 + i q_2$, we can rewrite the tautological one-form into $\alpha_0=\operatorname{Re}(\bar{p} dq)$. By substituting $q = z^k$ and $p = \frac{\omega}{k \bar{z}^{k-1}}$, we obtain $\operatorname{Re}(\bar{p} dq) = \operatorname{Re}(\bar{\omega} dz)$, thus $\omega_0(p,q)= \omega_0 (z,w)$.
		}

		For normalization purpose, we would prefer the conformal symplectic transformation
		\[\Phi: (z,w) \mapsto \left(q=z^k, p=\frac{w}{ \overline{z}^{k-1}}\right).
		\]
		{which is equivalent to making an additional inessential constant change of time}
		{which then} pulls the system
		
		\[\frac{|p|^2}{2} + \frac{s}{|q|^{2 -2/k}} + f=0
		\]
		back to
		\[
		\frac{|w|^2}{2 |z|^{2k-2}} + \frac{s}{|z|^{2k -2}} + f  = 0.
		\]
		On this energy level {we may now apply Lemma \ref{lem: time-reparametrization} and} multiply the Hamiltonian by the factor $ |z|^{2k-2}$ which just reparametrizes the flow on this energy hypersurface.
		With this we get
		\[
		\frac{|w|^2}{2} + f |z|^{2k -2} + s = 0.
		\]
		which is the system with Hamiltonian $\frac{|w|^2}{2} + f |z|^{2k -2} + s$ on its zero-energy level.
	\end{proof}
     {We remark that this has been used by McGehee for regularization purpose \cite{McGehee}}.
     
     \subsection{{Mechanical Billiards from Free Billiards}}
	We now draw our first consequences in the case $f = 0$. In this case one of the two systems is the system of free motions in the plane. It is classically known that a free billiard with a conic section as a reflection wall is integrable {\cite{Birkhoff}\cite{Tabachinkov}\cite{Kozlov-Treshchev}}. Namely, it allows elliptic, hyperbolic, parabolic, and line boundaries as integrable reflection walls. {We shall deduce this from our discussions on integrable Hooke/Kepler billiards in Section \ref{sec: duality_Hook_Kepler}, and include a direct proof for this fact in an Appendix \ref{sec: Joachimsthal}. }
	
	A conic section in the plane is described with six parameters as
	\begin{equation}
	\label{eq:conicsec}
	A z_1^2 + B z_1z_2 + C z_2^ 2 + D z_1 + E z_2  + F = 0,
	\end{equation}
	where all coefficients are real numbers and $A,B$, and $C$ are not all zero.
	{Since multiplication by a common factor to all the coefficients does not change the curve that it describes, only five out of the six parameters are free.}
	{In addition, when we identify conic sections which differ from each other just by scalings and rotations, then only three of the parameters are free.} 
 
	From this fact and Theorem \ref{thm: conformal_trans}, {we directly get the following proposition as an easy corollary.}
	
	\begin{prop}
		For any $k \in \mathbb{N}, k \geq 2$, the system $(\mathbb{C}, g_{flat}, 1 / |q|^{2-2/k},0)$  on zero-energy surface admit 5-parameter family of smooth integrable reflection walls without ruling out the scalings and the rotations, and 3-parameter family of smooth integrable reflection walls while ruling out the scalings and the rotations.
	\end{prop}

	{We may as well consider the case $s = 0$ which also gives rise to free motion.  With the same argument we get the following proposition:}
	\begin{prop}
		{For any $k \in \mathbb{N}, k \geq 2$, the system $(\mathbb{C}, g_{flat},   |z|^{2k -2},0)$ on the zero-energy surface admit 5-parameter family of smooth integrable reflection walls without ruling out the scalings and the rotations, and 3-parameter family of smooth integrable reflection walls {while} ruling out the scalings and the rotations.}
	\end{prop}

        	{We illustrate these propositions in the case $k = 2$.  The} complex square mapping $z \mapsto z^2 = q$ gives its lift 
	\begin{align*}
	&q_1 = z_1 ^2 - z_2^2 \\
	&q_2 = 2 z_1 z_2\\
	&p_1 = \frac{z_1w_1 {-} z_2 w_2}{z_1^ 2 + z_2^ 2} \\
	&p_2 = \frac{z_1w_2 {+} z_2 w_1}{z_1^ 2 + z_2^ 2}.
	\end{align*}
	after changing time-parametrization.

	Figure \ref{fig:transfomed_conicsection_k2} shows the reflection walls that are transformed from the  ellipses/hyperbolae
	\[
	\frac{(z_1 -  c_1)^2}{a^2} \pm  \frac{(z_2 -  c_2)^2}{b^2}=1
	\]
	by the mapping above $z \mapsto z^2$, in the case of $f=0$.
	\begin{figure}[p]
		\centering
		\begin{tabular}{cc}
			\begin{minipage}[t]{0.5\hsize}
				\includegraphics[keepaspectratio, height=4cm]{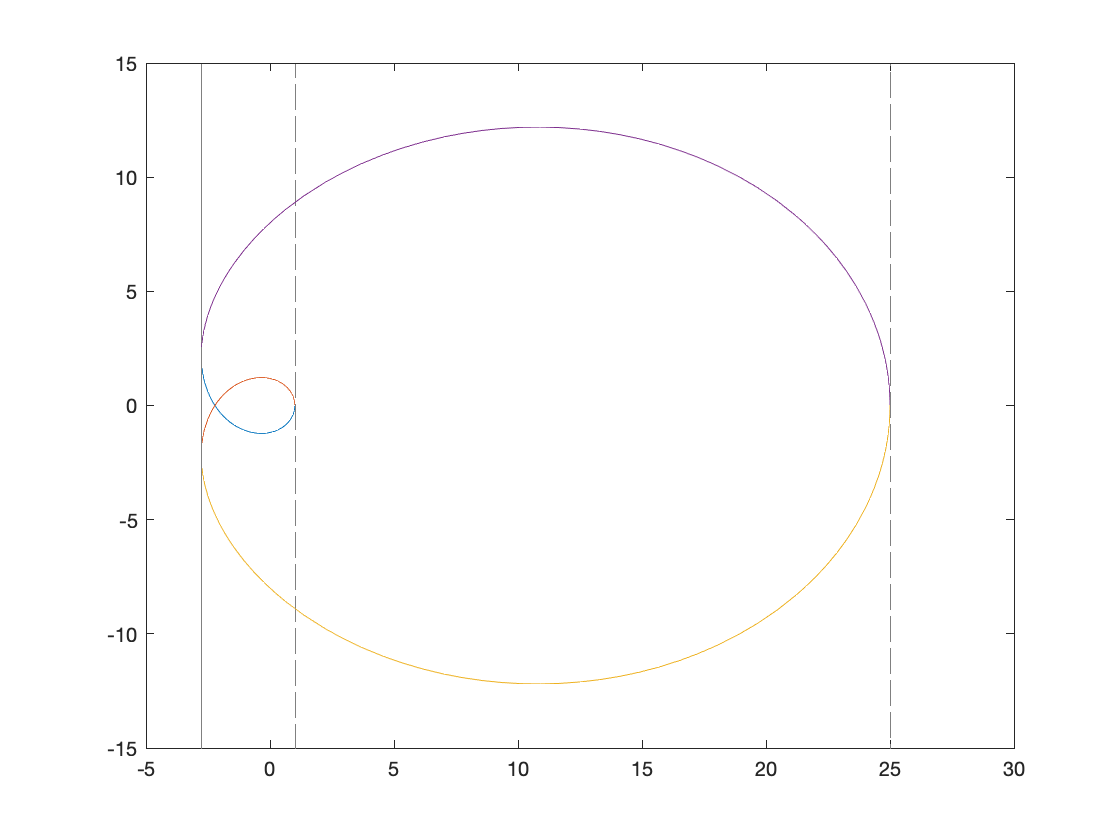}\\
				\vspace{-0mm}
				\centering
				\text{a. ellipse, $a = 3, b= 2, c_1 = 2, c_2 = 0$}
			\end{minipage}
			\begin{minipage}[t]{0.5\hsize}
				\includegraphics[keepaspectratio, height=4cm]{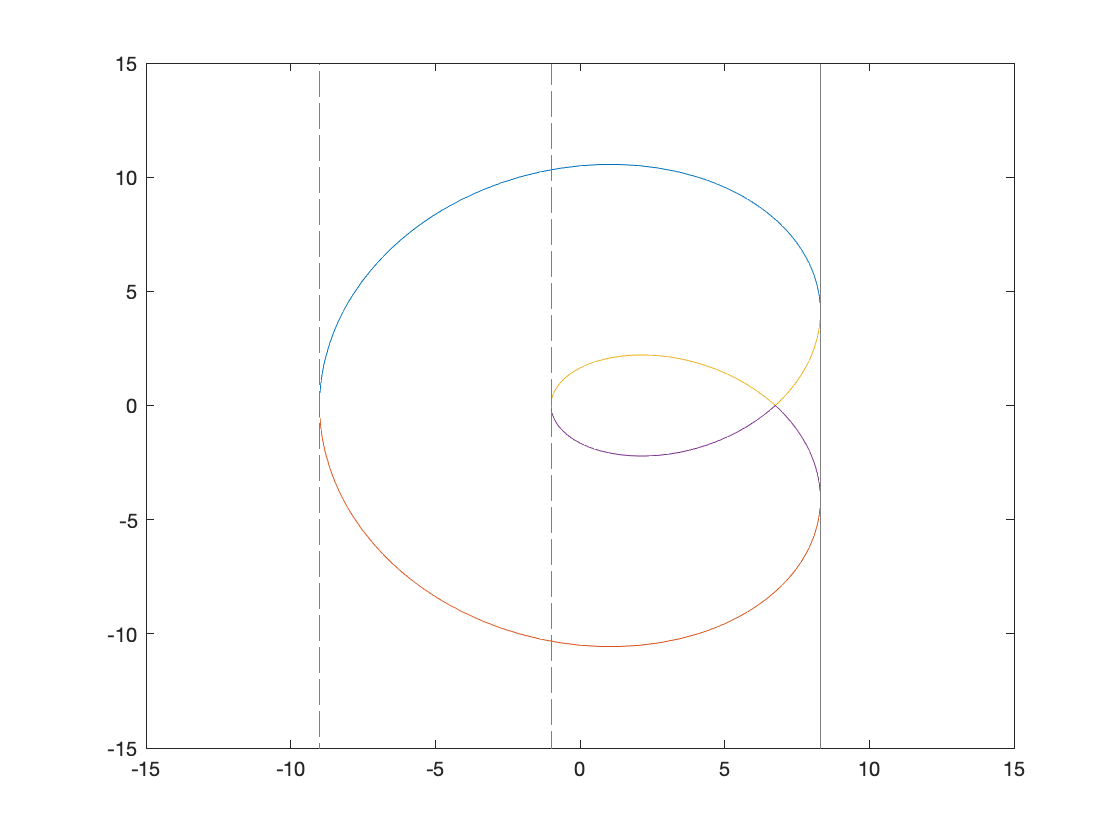}\\
				\hspace{3.5mm}
				\centering
				\text{b. ellipse, $a = 3, b= 2, c_1 = 0, c_2 = 1$}
			\end{minipage}
		\end{tabular}
		\begin{tabular}{cc}
			\begin{minipage}[t]{0.50\hsize}
				\centering
				\includegraphics[keepaspectratio, height=4cm]{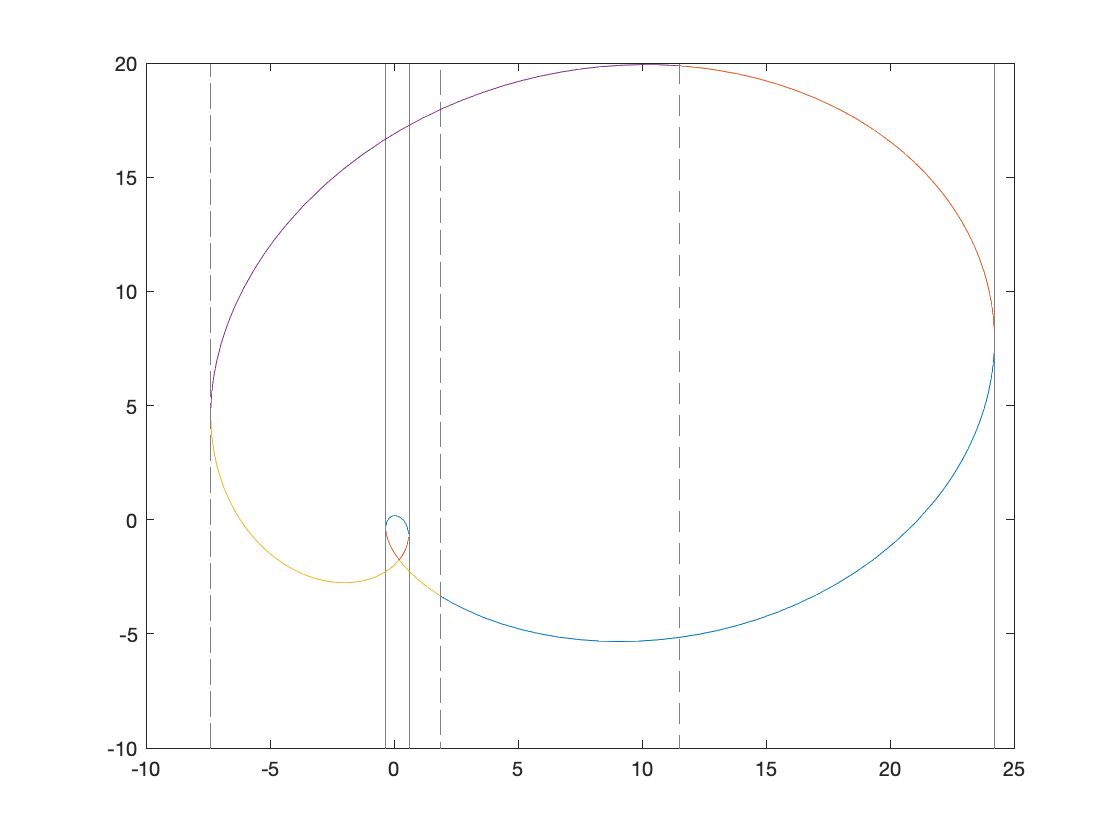}\\
				\centering
				\text{c. ellipse, $a = 3, b= 2, c_1 = 2, c_2 = 1$}\\
				\includegraphics[keepaspectratio, height=4cm]{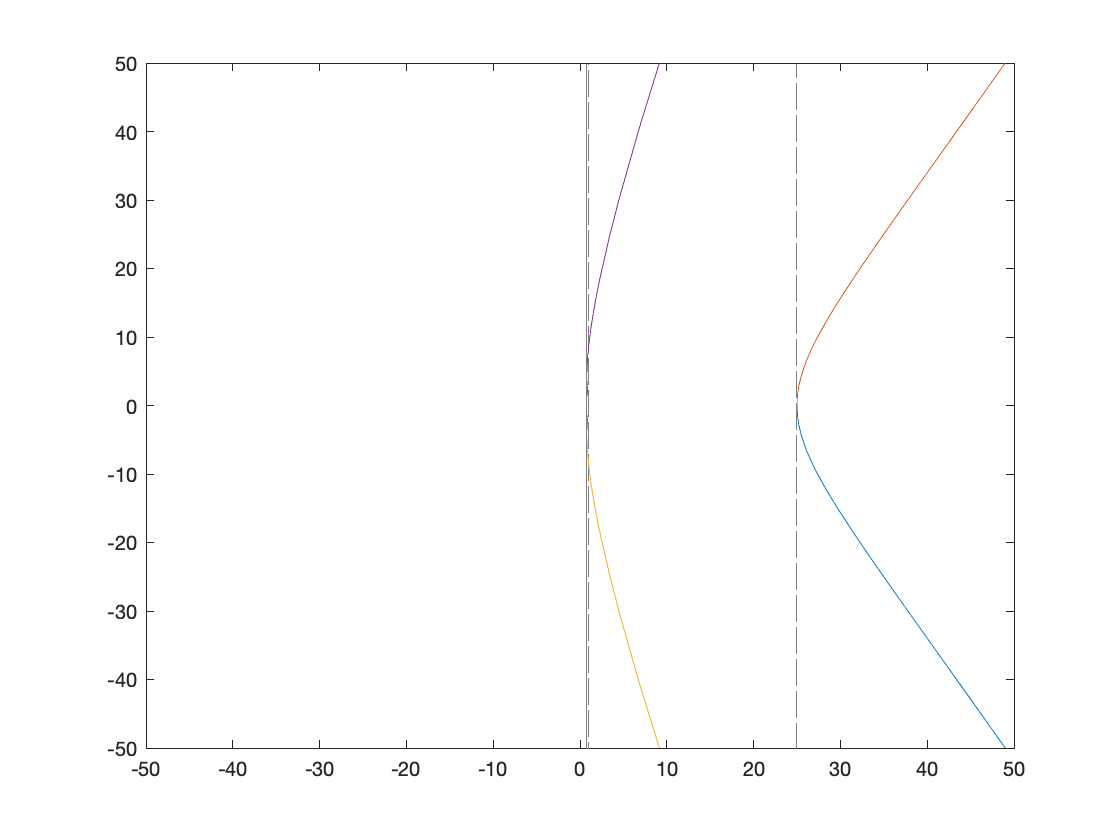}\\
				\centering
				\text{e. hyperbola, $a = 3, b= 2, c_1 = 2, c_2 = 0$}
				\includegraphics[keepaspectratio, height=4cm]{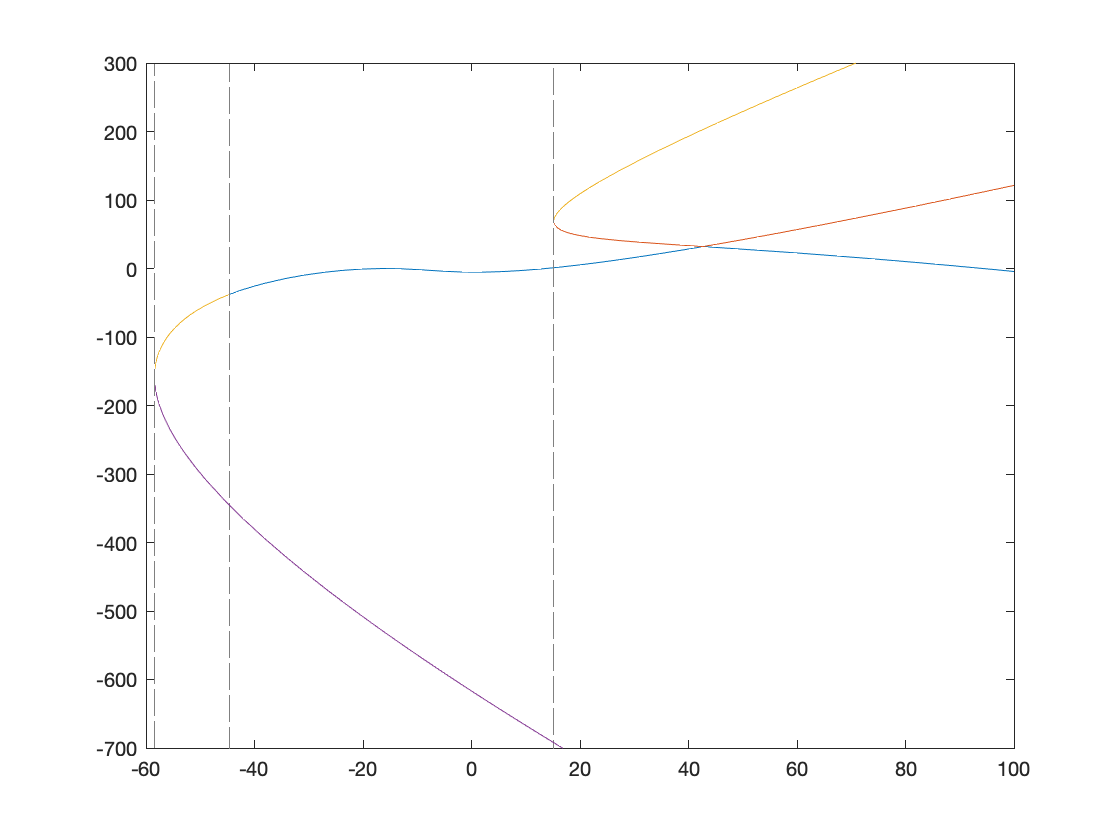}\\
				\centering
				\text{g. hyperbola, $a = 3, b= 2, c_1 = 3, c_2 = 4$}
			\end{minipage} &
			\begin{minipage}[t]{0.50\hsize}
				\centering
				\includegraphics[keepaspectratio, height=4cm]{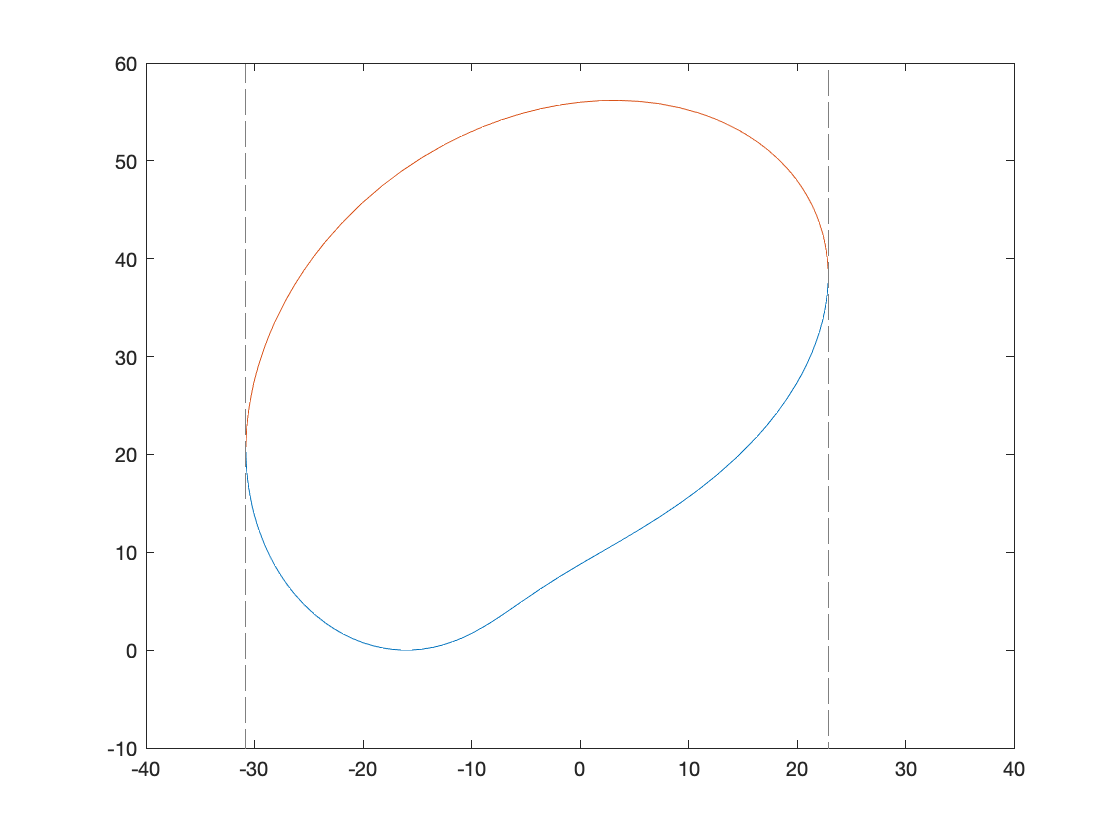}\\
				\centering
				\text{d. ellipse, $a = 3, b= 2, c_1 = 3, c_2 = 4$}\\
				\includegraphics[keepaspectratio, height=4cm]{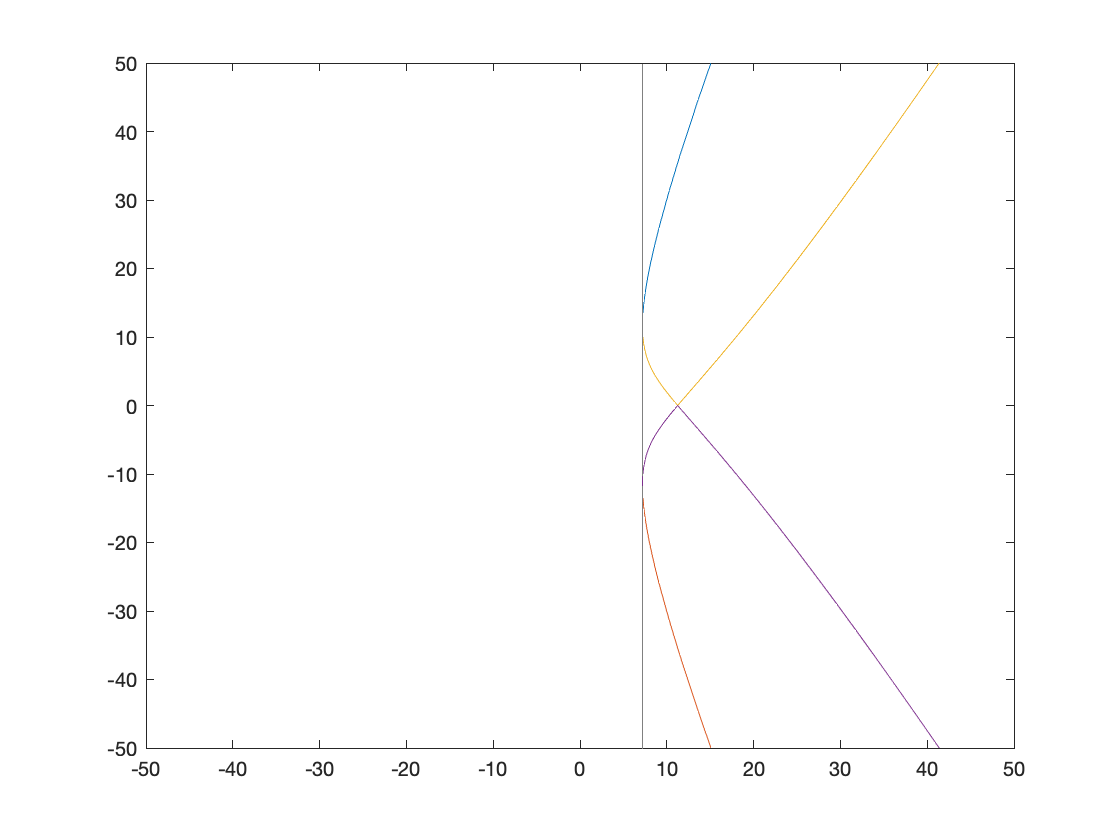}\\
				\centering
				\text{f. ellipse, $a = 3, b= 2, c_1 = 0, c_2 = 1$}
			\end{minipage}
		\end{tabular}
		\vspace{5mm}
		\centering
		\caption{transformed ellipses/hyperbolae by $z \mapsto z^2$}
			\label{fig:transfomed_conicsection_k2}
	\end{figure}
	{Notice that for a non-centrally symmetric curve, its centrally symmetric reflection is another branch of the pre-image of its image under the complex square mapping. Since centrally symmetric points are mapped to the same point under the complex square mapping which is locally a diffeomorphism, the image may thus have self-intersection points. For a non-centered ellipse, its image contains self-intersection points when the center of the ellipse is not too far away from the origin. The situation is exactly the contrary for a non-centered hyperbola: its image contains self-intersections points when its center is sufficiently far from the origin. } 
	
	{We shall provide a direct verification of the integrability of transformed curves when $s=0$ in Appendix \ref{sec: transformed_Jaochimsthal}.}
		
	{The law of reflections should also corresponds to each other through the conformal complex square mapping. This gives the following law of reflection at the target space of the mapping:  A particle is supposed to be reflected against the curve in the target space when the corresponding motions in the source space does so. Otherwise the particle just crosses the curve. Figure \ref{fig: law_reflection} illustrates this rule, in which the left picture shows what happens the source space, and the right picture in the target space. The dashed curve in the left picture is the centrally symmetric image of the reflection wall. A pair of centrally symmetric points lying in the reflection wall colored in green in the source space are mapped to one point which is an self-intersection point in the target space. A reflection point colored in red in the source space is mapped again to a reflection point in the target space, and vice versa. The blue points indicate where the particle just crosses without reflections in the source and the target spaces.}
	
	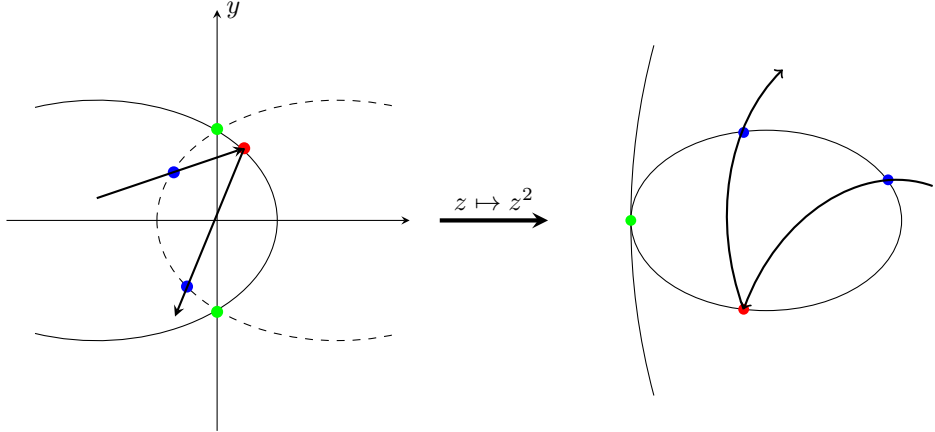
\begin{figure}
		\begin{tikzpicture}
		\begin{scope}[scale= 0.8]
		\draw[->,>=stealth] (-3.5,0)--(3.2,0); 
		\draw[->,>=stealth] (0,-3.5)--(0,3.5)node[right]{$y$}; 
		\draw[dashed](-1,0)arc(-180:-70:3 and 2);
		\draw[dashed](-1,0)arc(180:70:3 and 2);
		\draw[rotate=180](-1,0)arc(-180:-70:3 and 2);
		\draw[rotate=180](-1,0)arc(180:70:3 and 2);
		\coordinate (o) at (0, 0);
		\coordinate (A) at (0, 1.52);
		\fill[green] (A) circle (0.1);
		\coordinate (B) at (0, -1.52);
		\fill[green] (B) circle (0.1);
		\coordinate (R) at (0.45, 1.2);
		\fill[red] (R) circle (0.1);
		\coordinate (T1) at (-0.72, 0.8);
		\fill[blue] (T1) circle (0.1);
		\coordinate (T2) at (-0.5, -1.1);
		\fill[blue] (T2) circle (0.10);
		\draw[->,>=stealth,thick] (-2.0, 0.37)--(R);
		\draw[->,>=stealth,thick] (R)--(-0.7, -1.6);
		\draw[->,>=stealth, ultra thick](3.7,0) to node [above]{$z \mapsto z^2$} (5.5,0);
		\end{scope}
		\begin{scope}[xshift = 5.5cm, scale=0.6]
		\draw(0,0)arc(-180:180:3 and 2);
		\draw(0,0)arc(180:195:15);
		\draw(0,0)arc(-180:-195:15);
		\coordinate (O) at (0,0);
		\fill[green] (O) circle (0.12);
		\coordinate (R) at (2.5,-1.97);
		\fill[red] (R) circle (0.12);
		\coordinate (T1) at (5.7,0.9);
		\fill[blue] (T1) circle (0.12);
		\coordinate (T2) at (2.5,1.95);
		\fill[blue] (T2) circle (0.12);
		\draw[<-,rotate=-20, thick](R)arc(180:90:3 and 4);
		\draw[->,rotate=20,thick](R)arc(180:110:4 and 5);
		\end{scope}
		\end{tikzpicture}
		\caption{{the laws of reflections before and after the transformation induced by the complex square mapping}}
		\label{fig: law_reflection}
	\end{figure}

	{As in the case of the complex square mapping $z \mapsto z^2$, the two inverse branches are given respectively as
	\[
	z_1 =  \frac{q_2}{\sqrt{-2q_1 + 2\sqrt{q_1^2 + q_2^2}}}, ~ z_2 =  \frac{\sqrt{-2q_1 + 2\sqrt{q_1^2 + q_2^2}}}{2}
	\]
	and
	\[
	z_1 =  -\frac{q_2}{\sqrt{-2q_1 + 2\sqrt{q_1^2 + q_2^2}}}, ~ z_2 =  -\frac{\sqrt{-2q_1 + 2\sqrt{q_1^2 + q_2^2}}}{2}.
	\]}
	After eliminating {the} square roots in the equations, the quadratic curves given by the equation (\ref{eq:conicsec}) are mapped into some fourth-order equations of $q_1$ and $q_2$, so the transformed curve is more complicated and can be hard to identify via a more direct method. 
	
	\begin{figure}
		\centering
		\begin{tabular}{cc}
			\begin{minipage}[t]{0.5\hsize}
				\includegraphics[keepaspectratio, height=4cm]{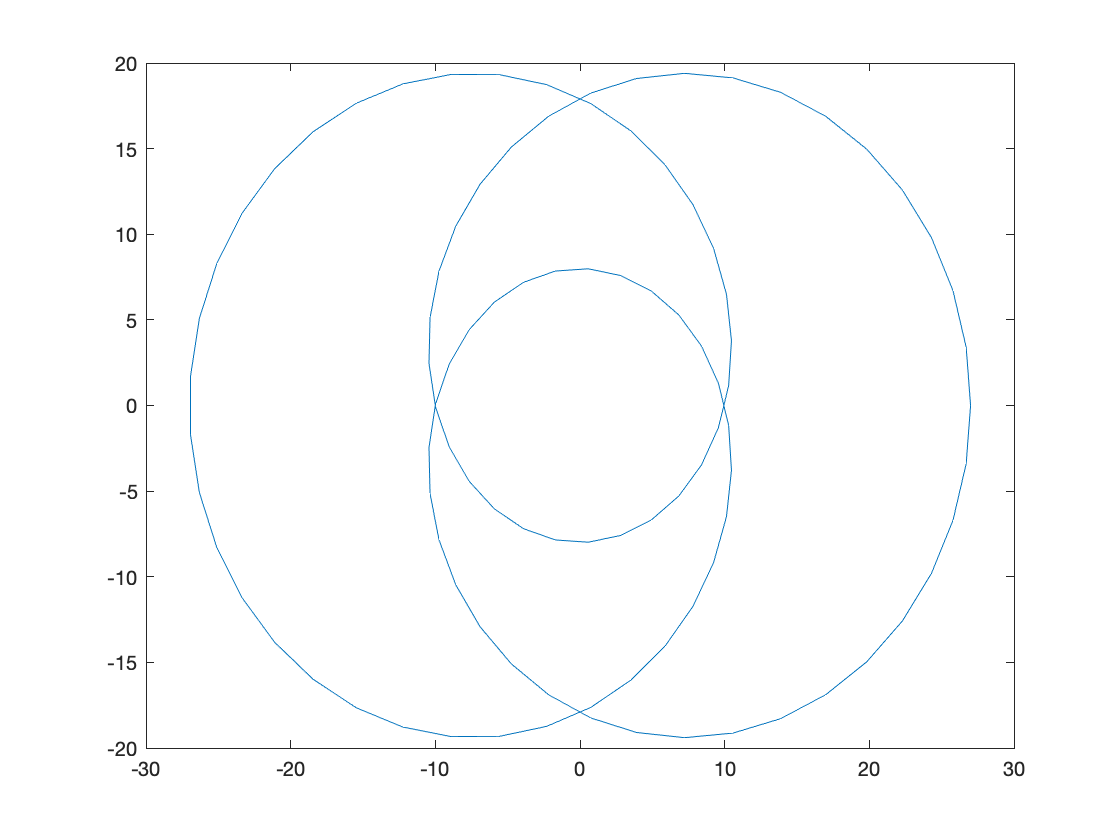}\\
				\vspace{-0mm}
				\centering
				\text{a. ellipse, $a = 3, b= 2, c_1 = 0, c_2 = 0$}
			\end{minipage}
			\begin{minipage}[t]{0.5\hsize}
				\includegraphics[keepaspectratio, height=4cm]{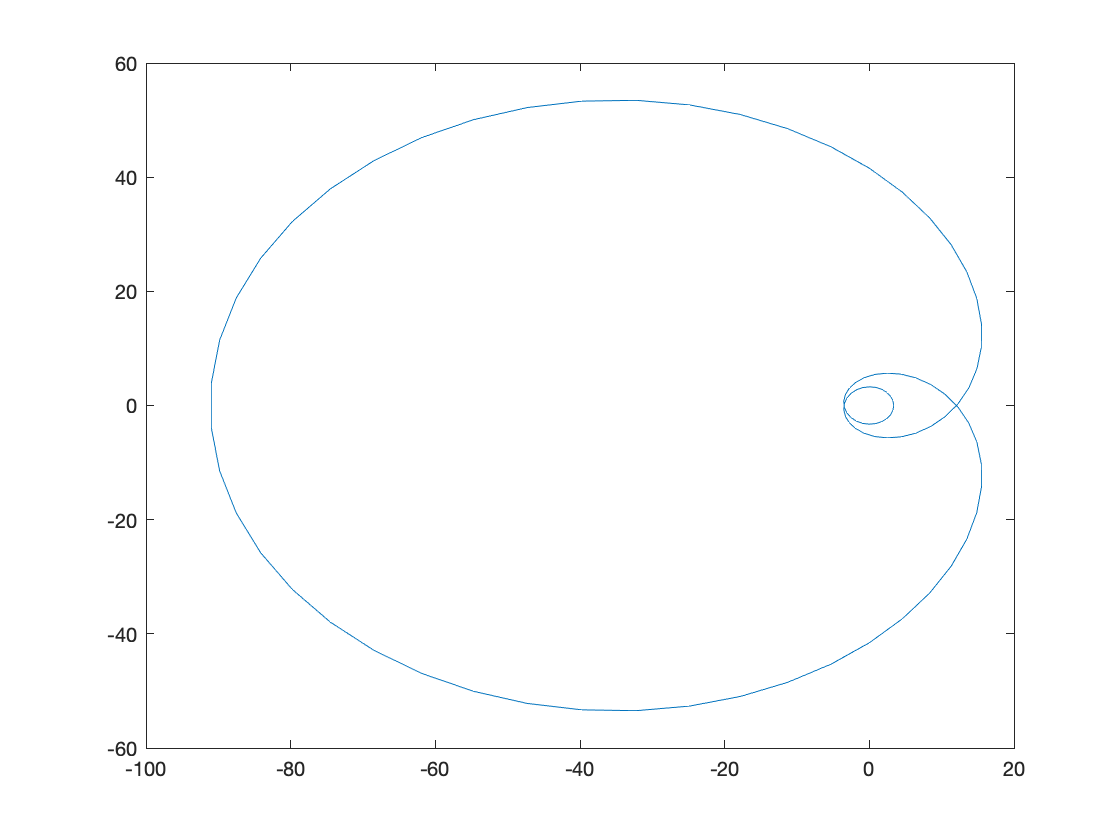}\\
				\vspace{-3.5mm}
				\centering
				\text{b. ellipse, $a = 3, b= 2, c_1 = 1.5, c_2 = 0$}
			\end{minipage}
		\end{tabular}
		\begin{tabular}{cc}
			\begin{minipage}[t]{0.50\hsize}
				\centering
				\includegraphics[keepaspectratio, height=4cm]{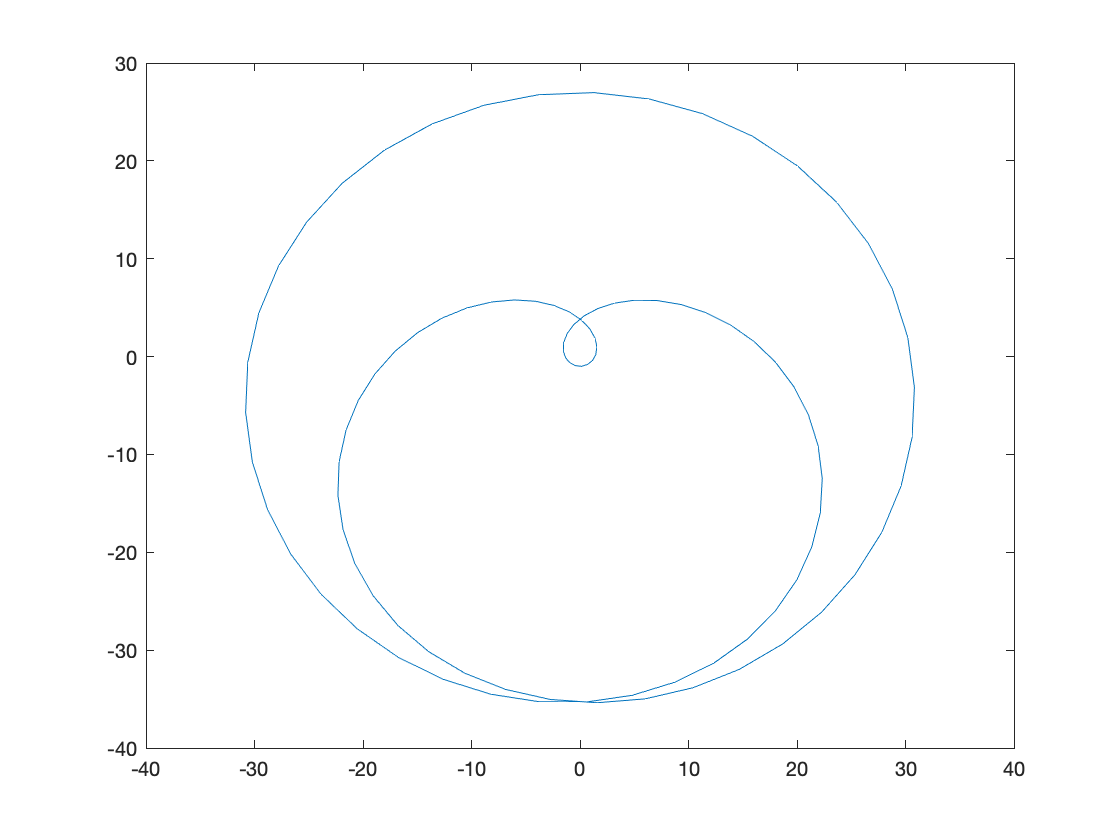}\\
				\centering
				\text{c. ellipse, $a = 3, b= 2, c_1 = 0, c_2 = 1$}\\
				\includegraphics[keepaspectratio, height=4cm]{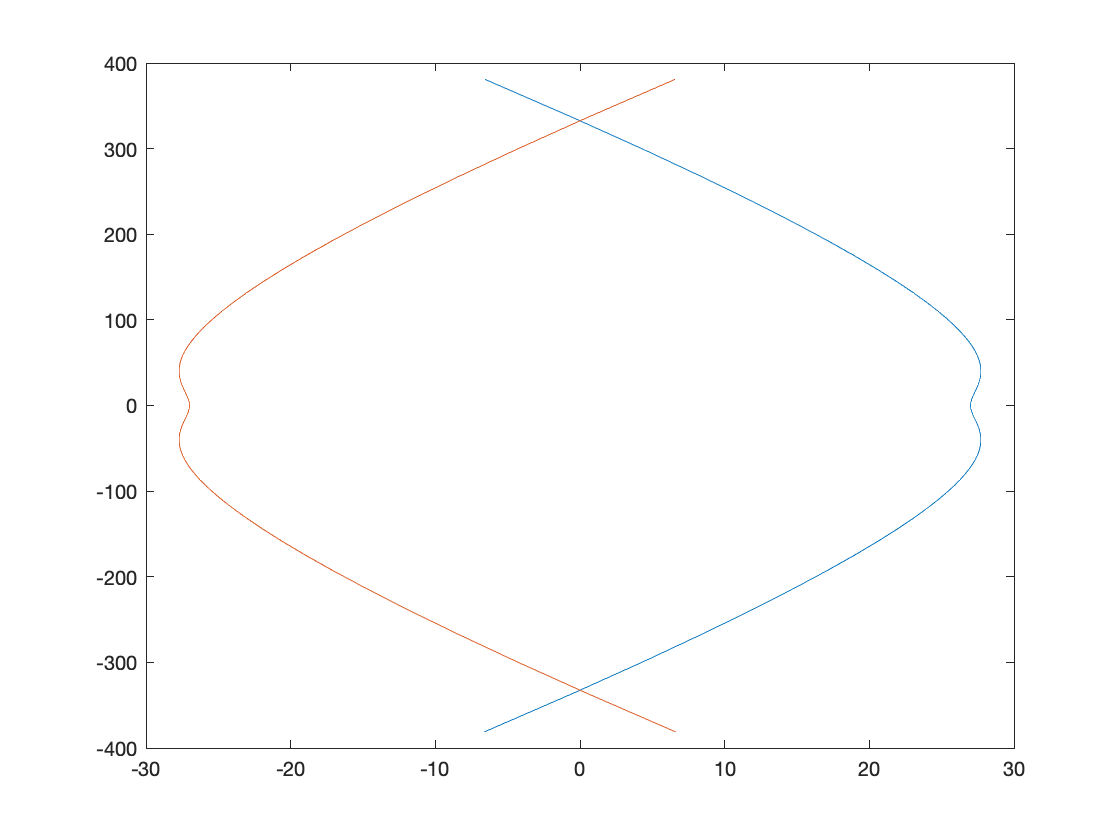}\\
				\centering
				\text{e. hyperbola, $a = 3, b= 2, c_1 = 0, c_2 = 0$}
				\includegraphics[keepaspectratio, height=4cm]{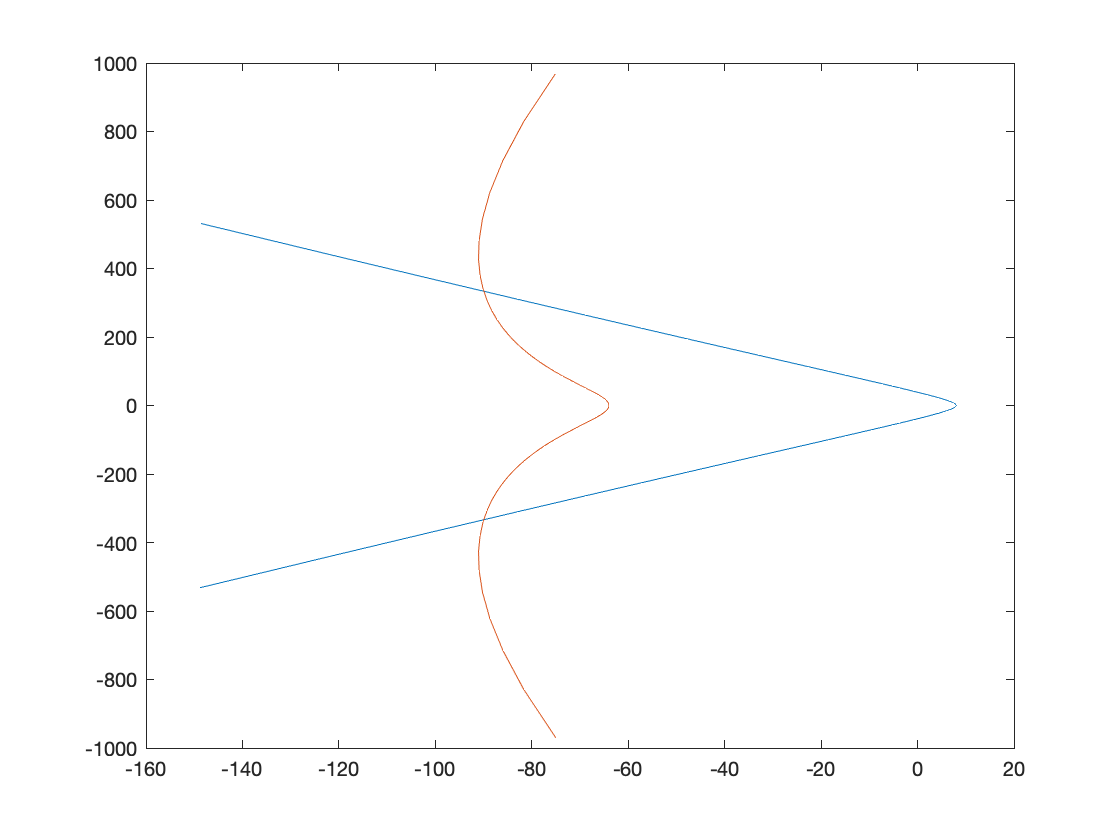}\\
				\centering
				\text{g. hyperbola, $a = 3, b= 2, c_1 = 1, c_2 = 0$}
			\end{minipage} &
			\begin{minipage}[t]{0.50\hsize}
				\centering
				\includegraphics[keepaspectratio, height=4cm]{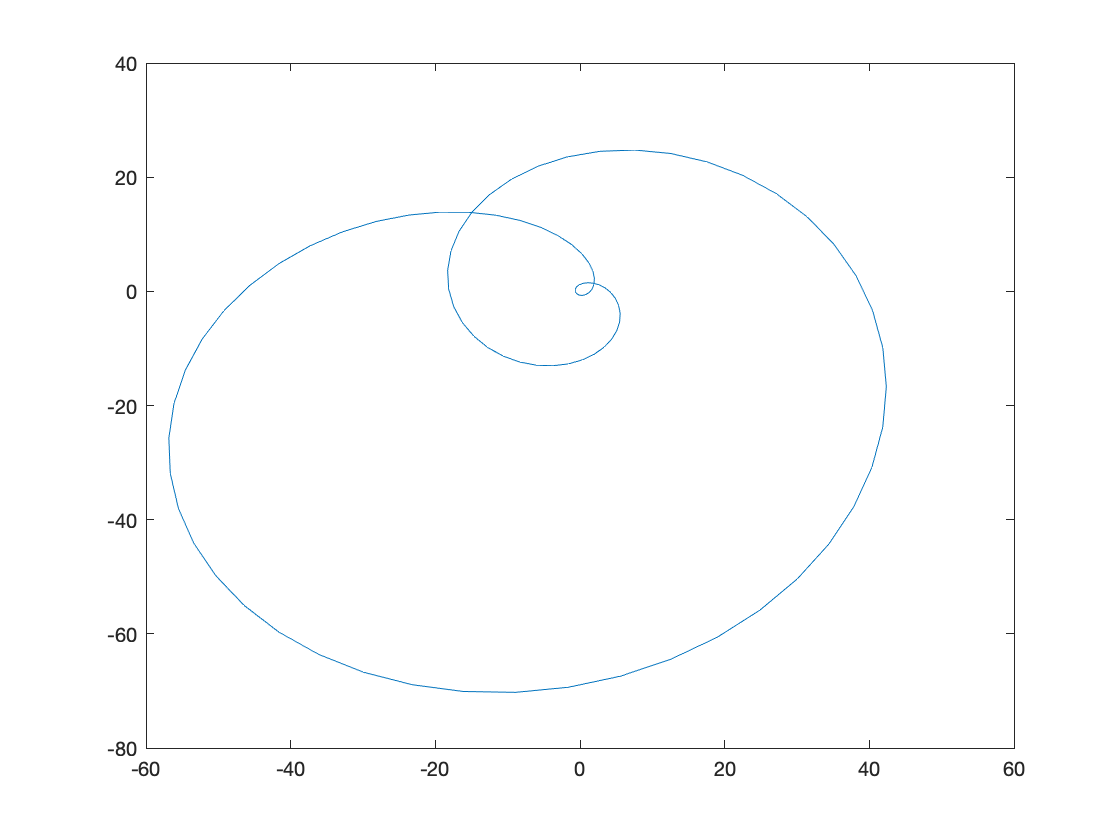}\\
				\centering
				\text{d. ellipse, $a = 3, b= 2, c_1 = 1, c_2 = 1$}\\
				\includegraphics[keepaspectratio, height=4cm]{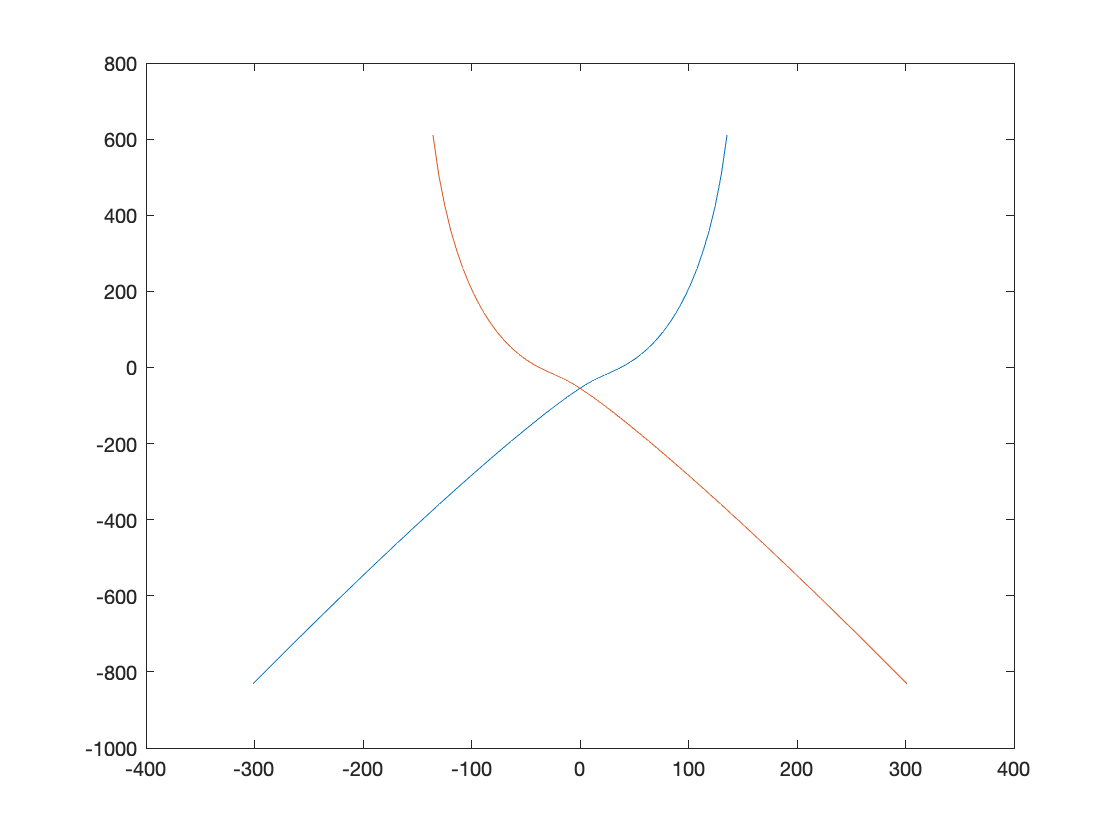}\\
				\centering
				\text{f. ellipse, $a = 3, b= 2, c_1 = 0, c_2 = 1$}
				\includegraphics[keepaspectratio, height=4cm]{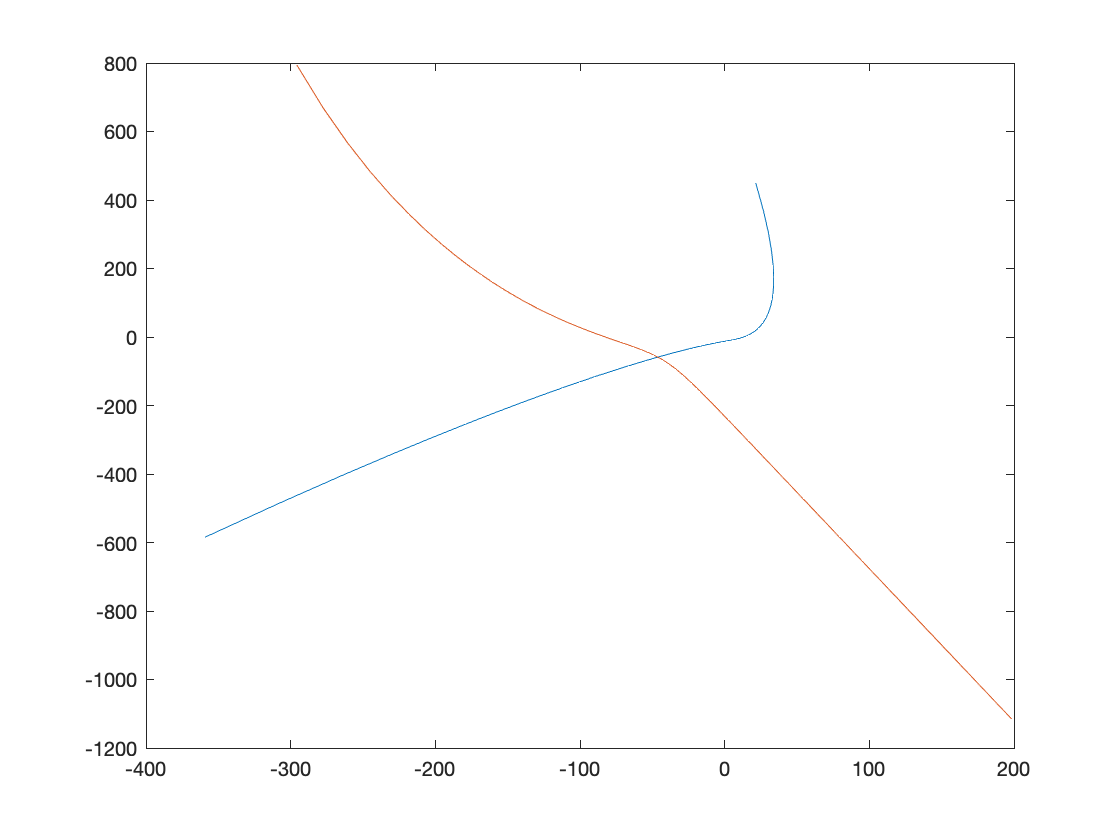}\\
				\centering
				\text{h. hyperbola, $a = 3, b= 2, c_1 = 1, c_2 = 1$}
			\end{minipage}
		\end{tabular}
		\vspace{5mm}
		\centering
		\caption{transformed ellipses/hyperbolae by $z \mapsto z^3$}
		\label{fig:transfomed_conicsection_k3}
	\end{figure}
	
	Next we consider the case $k=3$.
	Figure \ref{fig:transfomed_conicsection_k3} shows the image of ellipses/hyperbolae
	\[
	\frac{(z_1 -  c_1)^2}{a^2} \pm  \frac{(z_2 -  c_2)^2}{b^2}=1
	\]
	by the conformal mapping $z \mapsto  z^3=q$. One can see that the transformed curves from centered ellipses give four self-intersection points. As the center of an ellipse moves away from the origin, the number of self-intersection points is reduced to two{, and further to zero when the center is sufficiently far from the origin.} For hyperbolae, there always exist at least one self-intersection in the transformed curves. 
	
	
	{In the case of $fs \neq 0$, much less integrable reflection walls are known. In the next section we shall analyze the situation in the Hooke and Kepler problems.}
	
	\section{{Hooke, Kepler Billiards and their Dualities}}
	\label{sec: duality_Hook_Kepler}
	
	{Recall that by Hooke problem we refer to the mechanical system defined in the plane with force function $ f r^{2}$, in which $r$ is the distance of the particle to the center and $f\neq0$ is a real parameter. The force field is attractive/repulsive when $f$ is negative/positive. We accept both cases. Similarly by Kepler problem we refer to the mechanical system define in the plane with force function $s r^{-1}$ in which $s \neq 0$ is a real parameter. Again we accept both signs of $s$.}
	
	\subsection{{Integrable Hooke Billiards}}
	{Billiard systems defined with the Hooke problem is relatively well-studied and the following integrable reflection walls are known: Centered ellipses \cite{Jacobi Vorlesung}\cite{Fedorov}, lines, and combinations of confocal centered conic sections \cite{Pustovoitov2019}\cite{Pustovoitov2021}. Note that we call a conic section centered, when its center is at the origin. } 
	In the following theorem we provide a direct verification for the integrability of centerred conic section reflection walls in Hooke billiards.
	\begin{theorem}
		\label{thm: centered_conic}
		The attractive/repulsive Hooke billiard {$$(\mathbb{C}, g_{flat}, {-f |z|^2}, \mathcal{B})$$}
		with $\mathcal{B}$ being a centered conic section reflection wall is integrable.
	\end{theorem}
	\begin{proof}
		By {using the} rotational symmetry, we can write a centered conic ellipses as 
		\[
		F:=\frac{z_1^2}{a^2} + \frac{z_2^2}{b^2} -1=0.
		\]
		{Without loss of generality we can assume that $a\geq b$}.
		Let {$w:=(w_1,w_2)$} and {$w'=(w_1',w_2')$} denote {the linear} momenta, respectively, before and after the reflection against this centered elliptic reflection wall. 
		Set 
	         $F_1 := \partial F / \partial z_1$ and $F_2 := \partial F / \partial z_2$.
			
			{The normal vector at a point on $F=0$ is given by } 			\[
			n:=( F_1, F_2 ).
						\]
			{The normal component of $w$ is thus}
						\[
			{w_n:= \frac{w \cdot n}{|n|^2} n= \left( \frac{(w_1 F_1 + w_2 F_2 )F_1}{F_1^2 + F_2^2},  \frac{(w_1 F_1 + w_2 F_2 )F_2}{F_1^2 + F_2^2}   \right)}
			\]
			From the law of elastic reflection, the momenta after the reflection {$w'$} is 
			\[
			w' = w - 2w_n,
			\]
			{that is}
			\begin{align*}
			w_1' &= w_1 - \frac{2(w_1 F_1 + w_2F_2)F_1}{F_1^2 + F_2^2}\\
			&  =\frac{a^4 w_1 z_2^2 -2 a^2 b^2 w_2 z_1 z_2 - b^4 w_1 z_1^2}{a^4 z_2^2 + b^4 z_1^2},
			\end{align*}
			\begin{align*}
			w_2' &=w_2 - \frac{2(w_1 F_1 + w_2F_2)F_2}{F_1^2 + F_2^2}\\
			&=\frac{-a^4 w_2 z_2^2 -2 a^2 b^2 w_1 z_1 z_2 + b^4 w_2 z_1^2}{a^4 z_2^2 + b^4 z_1^2}.
			\end{align*}
		
		{We now search for possible first integrals of this Hooke billiard system.}
		{The Hamiltonian of the system is $|w|^2/2 + {f |z|^2}$, which admits three independent first integrals from its separability and rotational symmetry:}
		\[
		{2 f z_1^2 + w_1^2, \quad 2 f z_2^2 + w_2^2, \quad z_1 w_2 - z_2 w_1.}
		\]
		{So any combinations of (functions of) these first integrals are again an first integral.} We set
		\[
		{{\tilde{G}}(z_1,z_2,w_1,w_2) := k_1(2 f z_1^2 + w_1^2) + k_2 (2 f z_2^2 + w_2^2) + (z_1 w_2 - z_2 w_1)^2}
		\]
		{in which we square the angular momentum in order to have a quadratic function on $(w_{1}, w_{2})$ and $k_{1}, k_{2} \in \mathbb{R}$ are coefficients to be determined. }
		After the reflection at the reflection wall, $(z_1,z_2,w_1,w_2) $ is mapped to $(z_1,z_2, w'_1,w'_2)$.  {The} difference between the values {of ${\tilde{G}}$} before and after the reflection becomes 
		\begin{align*}
		&{\tilde{G}}(z_1,z_2,w_1,w_2) - {\tilde{G}}(z_1,z_2, w'_1,w'_2) = \\
		&\frac{4 z_1 z_2 (a^4 w_1 w_2 z_2^2 + a^2 b^2 w_1^2 z_1 z_2 - a^2 b^2 w_2^2 z_1 z_2 -b^4 w_1 w_2 z_1^2)}{a^4 z_2^2 + b^4 z_1^2 } \times\\
		&\frac{a^4 z_2^2 + a^2 b^2 z_1^2 - a^2 b^2 z_2^2 -b^4 z_1^2 - a^2 b^2 k_1 + a^2 b^2 k_2}{a^4 z_2^2 + b^4 z_1^2 }.
		\end{align*}
		Set 
		\[
		Z(z_1,z_2) := a^4 z_2^2 + a^2 b^2 z_1^2 - a^2 b^2 z_2^2 -b^4 z_1^2 - a^2 b^2 k_1 + a^2 b^2 k_2.
		\]
		{This is a factor in the numerator which does not depend on the momenta, and therefore its nullity implies preservation of $G$ under reflections.}
		
		{We deduce from $F=0$ that }
		\[
		Z = a^2 b^2 (a^2 - b^2 - k_1  +k_2).
		\] 
		Thus, $Z$ equals to $0$ if and only if 
		\begin{equation}
		\label{eq: condition1_k1k2}
		 k_1 -k_2=a^2- b^2 =
		 	a^{2} e^{2} \quad 
		\end{equation}
		in which $e$ is the eccentricity.
		 {After normalizing the coefficients, we thus get the following expression of the additional first integral as}
	         \begin{equation}
	         \label{eq: firstint_G}
	         {{G(z_1,z_2,w_1,w_2):=\dfrac{a^{2} e^{2}}{1+a^{2} e^{2}} (2f z_1^2 + w_1^2)+\dfrac{1}{1+a^{2} e^{2}}(z_1 w_2 - z_2 w_1)^2.}}
	         \end{equation}
		
		{The case of a centered hyperbola given by 
			\[
			\frac{z_1^2}{a^2} - \frac{z_2^2}{b^2} -1=0
			\]
		can be treated similarly and we get the condition on $k_1$ and $k_2$ as
		\begin{equation}
		\label{eq: condition2_k1k2}
		k_1-k_2 = a^2 + b^2= a^2 e^2.
		\end{equation}
		Thus, the first integral has 
		the same formula (\ref{eq: firstint_G}). 
	}
		
		 \end{proof}
	 
	 {From the conditions (\ref{eq: condition1_k1k2}) and (\ref{eq: condition2_k1k2}) on the coefficients $k_1$ and $k_2$, we can immediately deduce the integrability of reflection walls consist of confocal ellipses and hyperbolae.} 
	 \begin{cor}
	 	\label{cor: confocal_Hooke}
	{The attractive/repulsive Hooke billiard {$$(\mathbb{C}, g_{flat}, {-f |z|^2}, \mathcal{B})$$}
	with $\mathcal{B}$ being any combination of confocal {centered} ellipses and hyperbolae is integrable.} 
	 \end{cor}
 \begin{proof}
 	{Consider confocal centered ellipses in the form of
 	\[
 	\frac{z_1^2}{a^2} + \frac{z_2^2}{a^2 - c^2}=1, \quad a>c,
 	\]
 	and confocal centered hyperbolae in the form of 
 	\[
 	\frac{z_1^2}{b^2} - \frac{z_2^2}{ c^2- b^2}=1, \quad 0<b<c.
 	\]
 	For both cases, we obtain the same condition on the coefficients $k_1$ and $k_2$ such as
 	\[
 	k_1 - k_2 = c^2,
 	\]
 	thus there exists a common first integral given by 
 	\[
 	{G(z_1,z_2, w_1,w_2):= \dfrac{c^2}{1+c^2} (2f z_1^2 + w_1^2)+\dfrac{1}{1+c^{2}}(z_1 w_2 - z_2 w_1)^2.}
 	\]}
\end{proof}
	
		 {We note that by a direct limiting procedure we also get the integrability of the Hooke billiard with a line or any combination of parallel/perpendicular lines as an integrable boundary. }
		 \begin{cor}
		 	\label{cor: lines_Hooke}
		 	{The attractive/repulsive Hooke billiard {$$(\mathbb{C}, g_{flat}, {-f |z|^2}, \mathcal{B})$$}
		 	with $\mathcal{B}$ being {any combination of parallel/perpendicular lines is integrable. }} 
		 \end{cor}
	\begin{proof}
		{When we take a limit $e \to \infty$ in the first integral (\ref{eq: firstint_G}), we obtain the form:		
		\[
		{w_1^2 + 2f z_1^2}
		\]
		which is invariant under {reflections} against a line which is parallel to $z_1-$ or $z_2-$axis, {as well as for any {combinations} of lines which are parallel to the $z_1-$ or $z_2-$axis. }
		}
	\end{proof}
	
	\begin{rem}
		The Hooke billiard also allows other integrable reflection walls. As a more or less trivial example, any combination of lines passing through the center are integrable with the first integral $(z_1 w_2 - z_2 w_1)^2$.
	\end{rem}

	\subsection{{Integrable Kepler Billiards}}
	{{As opposed to the study of Hooke billiards,} the study of Kepler billiard, on the other hand, seems to be rather recent. In \cite{Boltzmann}, L. Boltzmann considered the billiard system of a central force problem in the plane, {which includes the Kepler problem, }with a line as wall of reflection. He asserted that {any billiard system} thus obtained is ergodic.  Recently Gallavotti-Jauslin \cite{Gallavotti-Jauslin} disproved this assertion in the case that the central force problem is the Kepler problem, {by actually showing the integrability of the corresponding billiard system.} This integrability is revisited together with an in-detailed analysis on its integrable dynamics in Felder \cite{Felder}. An alternative proof of this integrability based on projective invariance of the Kepler problem is provided in \cite{Zhao}}. 
	
	{We now make a revisit of the first integral of Gallavotti-Jauslin $G$ {using} the complex square mapping, thus provide yet another alternative proof for the integrability of Gallavotti-Jauslin \cite{Gallavotti-Jauslin} {as well as} some extensions.}
	
	\begin{lemma}
		\label{lem: Gallavotti-Jauslin}
		{The additional first integral
		\[
		{G(z_1,z_2,w_1,w_2):=\dfrac{a^{2} e^{2}}{1+a^{2} e^{2}} (2f z_1^2 + w_1^2)+\dfrac{1}{1+a^{2} e^{2}}(z_1 w_2 - z_2 w_1)^2}
		\]
		 given in Theorem \ref{thm: centered_conic} {of the Hooke billiard $(\mathbb{C}, g_{flat},  {-f |z|^2}, \mathcal{B})$}
		  is transformed, {after multiplying by $(1+a^{2} e^{2})$,} under the complex square mapping $\mathbb{C} \setminus O \to \mathbb{C} \setminus O,  z \mapsto z^2$} into {Gallavotti-Jauslin's first integral}
		 \begin{align*}
		 A(p_1,p_2,q_1,q_2) := {(p_1 q_2 - p_2 q_1)^2 }
		 - 2 \tilde{a} \left((-p_1 q_2 + p_2 q_1)p_2 - \frac{{s} q_1}{\sqrt{q_1^2 + q_2^2}}\right)
		 \end{align*}
		{in which {$\tilde{a}=a^{2} e^{2}/2$}, {on the $-f$-energy hypersurface} of the Kepler problem
		$
		\left(\mathbb{C} \backslash O, g_{flat}, \frac{s}{ {|q|}} \right) 
		$.
		}
	\end{lemma}
   \begin{proof} 
	{We observe that $(z_1 w_2 - z_2 w_1)^2$ is mapped into the squared angular momentum $(q_1 p_2 - q_2 p_1)^2$. Indeed, we see that
	\begin{align*}
	z_1 w_2 - z_2 w_1 
	&= \frac{(z_1^2 - z_2^2)(z_1 w_2 + z_2 w_1)- 2z_1z_2(z_1w_1 - z_2w_2)}{z_1^2 + z_2^2}\\
	&= q_1 p_2 - q_2 p_1.
	\end{align*}
	}
	
	 {We now consider the other term {$2f z_1^2 +  w_1^2$. }
	 With the relations 
	 $$q_1 = z_1^2 - z_2^2, \quad q_2 = 2z_1 z_2,$$
	  we get
	 \begin{align*}
	 z_1^2 
	 &= \frac{z_1^2 -z_2^2 + \sqrt{(z_1^2 - z_2^2)^2 + 4z_1^2 z_2^2}}{2}\\
	 &= \frac{q_1 + \sqrt{q_1^2 + q_2^2}}{2}, 
	 \end{align*}
	 and 
	 \begin{align*}
	 z_2^2 &=z_1^2 - q_1 \\
	 &= \frac{-q_1 + \sqrt{q_1^2 + q_2^2}}{2}.
	 \end{align*}
	 From these and $w_1= z_1 p_1 + z_2 p_2$, we get
	 \begin{align*}
	 w_1^2 &= (z_1 p_1 + z_2 p_2)^2 \\
	 &= z_1^2 p_1^2 + 2 z_1z_2 p_1p_2 + z_2^2 p_2^2 \\
	 &=  \frac{1}{2} \left( p_1^2 q_1 - p_2^2 q_1 + 2 p_1 p_2 q_2 {+}  p_1^2 \sqrt{q_1^2 + q_2^2}  {+} p_2^2 \sqrt{q_1^2 + q_2^2}\right).
	 \end{align*}
	 From these we see that {$2f z_1^2 +  w_1^2$} is mapped into 
	\[
	{f\left(q_1 {+} \sqrt{q_1^2 + q_2^2}\right) + \frac{1}{2} \left( p_1^2 q_1 - p_2^2 q_1 + 2 p_1 p_2 q_2 {+}  p_1^2 \sqrt{q_1^2 + q_2^2}  {+} p_2^2 \sqrt{q_1^2 + q_2^2}\right).}
	\]
	{After fixing the energy of the Hooke problem to $s$ and transforming the resulting system via the complex square mapping we get the energy constraint 
	$$ (p_1^2 + p_2^2)/2 - s/(\sqrt{q_1^2 + q_2^2}) + f = 0.$$
	From which we deduce that 
	\begin{align*}
	&{f}\left(q_1 {+} \sqrt{q_1^2 + q_2^2}\right) + \frac{1}{2} \left( p_1^2 q_1 - p_2^2 q_1 + 2 p_1 p_2 q_2 {+}  p_1^2 \sqrt{q_1^2 + q_2^2}  {+} p_2^2 \sqrt{q_1^2 + q_2^2}\right)\\
	=&-\left((-p_1 q_2 + p_2 q_1)p_2 -s \cdot \frac{q_1}{\sqrt{q_1^2 + q_2^2}} - s\right).
	\end{align*}}
	Therefore, the additional first integral
	\[
	G(z_1,z_2,w_1,w_2) =  a^2 e^2({2f} z_1^2 + w_1^2) +(z_1 w_2 - z_2 w_1)^2
	\]
	is transformed into the form 
	\begin{align*}
	A(p_1,p_2,q_1,q_2) := {(p_1 q_2 - p_2 q_1)^2 }
	- 2 \tilde{a} \left((-p_1 q_2 + p_2 q_1)p_2 - \frac{{s} q_1}{\sqrt{q_1^2 + q_2^2}}\right),
	\end{align*}
	where {$\tilde{a}=a^{2} e^{2}/2$} ({which is the distance from the center to one of the foci in case when the transformed curve is a focused ellipse or hyperbola}),{ on the $-f$-energy hypersurface of the Kepler problem.}
	}
	\end{proof}
	
	{We say a conic section is \emph{focused}, when the origin is a focus of it. Using the duality between the Kepler billiard system and the Hooke billiard system given in the {Theorem \ref{thm: duality}}, we deduce various integrable Kepler billiards {from Theorem \ref{thm: centered_conic}}, that we summarize in the following theorem:} 
	
	\begin{theorem}
		\label{thm: focused_conic}
		The Kepler system $(\mathbb{C}, g_{flat}, \frac{{s}}{|q|}, \mathbb{R})$
		admits {any focused conic sections, {degenerate cases allowed, }as {integrable reflection walls}}. {These include}
		\begin{enumerate}
			\item any focused parabola
			\item any focused ellipse
			\item {any focused hyperbola}
			\item any line.
		\end{enumerate} 
		{The additional first integral is given in Lemma \ref{lem: Gallavotti-Jauslin}.}
	\end{theorem}
	\begin{proof}
		{We discuss case by case.} 
		\begin{enumerate}
		\item  {Any lines are integrable reflection wall for the Hooke potential. By rotation-invariance it is enough to consider the case of {a line given by the expression $z_1 = c, c \in \mathbb{R} \backslash \{0\}$}, which is transformed by $z \mapsto z^2 = q$ into the parabola
		\[
		q_1 = -\frac{q_2^2}{4c^2} + c^2,
		\]
		focused at the origin. }
		
		\item  {Consider a centered ellipse} given by 
		\begin{equation}
		\label{eq: centered_ellipse}
		\frac{z_1^2}{a^2} + \frac{{z_2^{2}}}{b^2} = 1
		\end{equation}
		We parametrize this elliptic curve as 
		\[
		z_1 = a \cos u, \quad z_2 = b \sin u
		\]
		with a parameter $u \in [0, 2 \pi)$. Then the image of this curve by the conformal mapping $z \mapsto z^2= q$ is given by
		\[
		{q_{1}} = a^2 \cos^2 u - b^2 \sin^2 u, \quad {q_{2}} = 2ab \sin u \cos u.
		\] 
		{which describes the focused ellipse}
		\begin{equation}
		\label{eq: trans_focused_ellipse}
		\frac{(q_1 - (a^2 - b^2)/2)^2}{(a^2 + b^2)^2/4} + \frac{q_2^2}{a^2 b^2} =1.
		\end{equation}
		
		\item  
		 Consider a centered hyperbola given by 
		\begin{equation}
		\label{eq: centered_hyperbola1}
		\frac{z_1^2}{a^2} - \frac{z_2^2}{b^2} = 1, \quad  {a \neq b},
		\end{equation}
		{{parametrized as 
		\[
		z_1 =a \cosh u, \quad z_2 = b \sinh u
		\]}
		{with parameter $u \in (-\pi, \pi)$ for one branch and  $u \in (-\pi, \pi)$ for the other branch.
		Then the image of this curve by the conformal mapping $z \mapsto z^2= q$ is given by
		$$q_{1} = a^2 \cosh^2 u - b^2 \sinh^2 u, \quad q_{2} = 2 a b \sinh u \cosh u.$$}}
		
	{We thus get that the transformed curve satisfies
		\begin{equation}
		\label{eq: trans_focused_hyp}
		\frac{(q_1 - (a^2 + b^2)/2)^2}{(a^2 - b^2)^2/4} - \frac{q_2^2}{a^2 b^2} =1.
		\end{equation}
    		which describes a focused hyperbola. Indeed this image is seem to be a branch of this hyperbola. The pre-image of the other branch of this hyperbola is the confocal centered hyperbola given by }
		\begin{equation}
		\label{eq: centered_hyperbola2}
		\frac{z_1^2}{b^2} - \frac{z_2^2}{a^2} = 1, 
    		\end{equation}
		{To see this, it is enough to exchange the roles of $a$ and $b$ in the above reasoning. }
		
		{Since the pre-image of the focused hyperbola consists of two confocal hyperbolae, we may thus conclude with Corollary \ref{cor: confocal_Hooke}. }

		\item  {Finally, a hyperbola given by
		\[
		\frac{z_1^2}{a^2} - \frac{z_2^2}{a^2} = 1
		\]
		is transformed by the conformal mapping $z \mapsto z^2=q$ into the line 
		\[
		{q_1} = {a^{2}}.
		\]
	}
	\end{enumerate}
	\end{proof}

	{In appendix \ref{sec: G-J}, for the purpose of comparison, we directly verify the invariance of Gallavotti-Jauslin's first integral in the case that the reflection walls are focused ellipses or focused hyperbola.}
	
	
	\begin{cor}
		\label{cor: confocal_Kepler}
		{The Kepler system $(\mathbb{C}, g_{flat}, \frac{{s}}{|q|}, \mathbb{R})$
		admits any combination of confocal focused ellipses and hyperbolae as an integral reflection wall.	}
	\end{cor}
	\begin{proof}
		{It suffices to see that confocal centered ellipses/hyperbolae are transformed into confocal focused ellipses/hyperbolae by the complex square mapping $z \mapsto z^2$. This can be easily checked from the forms of a transformed focused ellipse (\ref{eq: trans_focused_ellipse}) and a transformed focused hyperbola (\ref{eq: trans_focused_hyp}) by setting $b^2 = a^2 -c^2$ for ellipses and $b^2 = c^2 -a^2$ for hyperbolae. }
	\end{proof}

	{Similarly, from Corollary \ref{cor: lines_Hooke} we obtain integrable Kepler billiards with any combination of  focused parabolae with collinear major axis.}
	
	\begin{cor}
	\label{cor: confocal_parabola_halfline_Kepler}
	{The Kepler system $(\mathbb{C}, g_{flat}, \frac{{s}}{|q|}, \mathbb{R})$
	admits any combination of focused parabolae with collinear major axises as an integrable reflection wall.}
	\end{cor}
	\begin{proof}
		{The argument follows directly from Theorem \ref{thm: focused_conic}, Case 1 and Corollary \ref{cor: lines_Hooke}.  }
	\end{proof}
	
	\subsection{{From Hooke/Kepler Billiards to Free Billiards}}

	{We now discuss the classical case of free billiards based on our discussions on integrable Hooke/Kepler billiards, by setting $f=0$ in the Hooke billiards, or $s=0$ in the Kepler billiards. The following proposition now becomes a direct corollary.}

	\begin{cor}
		\label{prop:free_conic}
		{Free billiards admit conic section reflection walls as integrable reflection wall. } 
	\end{cor}
	
	{We now link the additional first integral given by (\ref{eq: firstint_G}) to the well-known Joachimsthal first integral, as follows:}
	From Theorem \ref{thm: centered_conic}, in the case of $f =0$, we have the additional first integral 
	\[
	(a^2 -  b^2 )w_1^2 + (z_1 w_2 - z_2 w_1)^2
	\]
	for the free billiard with a centered elliptic integrable reflection wall given by 
	\begin{equation}
	\label{eq: centered_ellipse2}
	\frac{z_1^2}{a^2} + \frac{z_2^2}{b^2} =1.
	\end{equation}

		By dividing this by $a^2 b^2$,we get
		\begin{align*}
		& \phantom{=}\left( \frac{1}{b^2} - \frac{1}{a^2} \right)w_1^2 + \frac{(z_1 w_2 - z_2 w_1)^2}{a^2 b^2} \\
		& =
		\frac{w_1^2}{b^2} - \frac{w_1^2}{a^2} +\frac{z_1^2 w_2^2 - 2 z_1 z_2 w_1w_2 + z_2^2 w_1^2}{a^2 b^2}\\
		&= \frac{w_1^2}{b^2} - \frac{w_1^2}{a^2} + \frac{z_2^2 }{b^2 }\cdot \frac{w_1^2}{a^2} + \frac{z_1^2 }{a^2 }\cdot \frac{w_2^2}{b^2} - \frac{2 z_1 z_2 w_1w_2 }{a^2 b^2}\\
		&= \frac{w_1^2 + w_2^2}{b^2} - \left( 1 - \frac{z_2^2}{b^2} \right) \frac{w_1^2}{a^2} - \left( 1 - \frac{z_1^2}{a^2} \right) \frac{w_2^2}{b^2}  - \frac{2 z_1 z_2 w_1w_2 }{a^2 b^2}\\
		&= \frac{1}{b^2} - \left( \frac{z_1^2 w_1^2}{a^4}  +  \frac{z_2 w_2}{b^4} +  \frac{2 z_1 z_2 w_1w_2 }{a^2 b^2} \right)\\
		&= \frac{1}{b^2 } - \left( \frac{z_1 w_1}{a^2} + \frac{z_2 w_2}{b^2} \right)^2.
		\end{align*}
		In the fourth equation, we used the equation of centered ellipse (\ref{eq: centered_ellipse2}).
		{In which we recognize} the classical {Joachimsthal} first integral 
		$$\dfrac{z_1 w_1}{a^2} + \dfrac{z_2 w_2}{b^2} $$
		 of the {free} billiard with an elliptic boundary. 
	
	\subsection{{Conjectures related to the Birkhoff Conjecture}}
	{From Theorem \ref{thm: centered_conic} and in view of {the} Birkhoff-Poritsky's conjecture, we make the following conjectures for Hooke and Kepler billiards. 
	\begin{conjecture} 
		{The only Hooke billiards with smooth connected reflection walls {which are integrable on all regular energy hypersurfaces} are those with a branch of a centered conic section or a line.} 	\end{conjecture}
	}

	\begin{conjecture} 
		{The only Kepler billiards with smooth connected reflection walls {which are integrable on all regular energy hypersurfaces} are those with a focused conic section or a line.}
	\end{conjecture}
	\section{{Integrable Stark-type Billiards}}
	\label{sec: Stark} 
	\subsection{Separability and Integrability of Stark-type Billiards}
	 {In this section, we investigate some two degrees of freedom mechanical systems which are separable after the complex square mapping and integrable reflection walls for such systems. We consider some special class of systems with force function} given in the form of 
	\[
	{\frac{s}{|q|} + V(q), \quad V \in {C^{\infty}(\mathbb{R}^2 \setminus O, \mathbb{R})}.}
	\]
	{so that the Kepler problem is further modified by the additional influence from $V(q)$.} The Hamiltonian of such {a system is}
	\begin{equation}
	\label{eq: original_H}
	H = \frac{|p|^2}{2} {-} \frac{s}{|q|} - V(q_1, q_2).
	\end{equation}
	{On its fixed energy hypersurface $\{H + f = 0\}$ we may again transform the system by the complex square mapping after a proper time change as described in Theorem  \ref{thm: conformal_trans} which then leads to the system
		\[
	\hat{H} = \frac{|w|^2}{2} {-} s + f (z_1^2 + z_2^2) {-}(z_1^2 + z_2^2)V(z_1^2 - z_2^2, 2 z_1 z_2).
	\]}
	Now the transformed Hamiltonian $\hat{H}$ is separable in $(z_1, z_2)$ coordinates if and only if the term
	\[
	(z_1^2 + z_2^2)V(z_1^2 - z_2^2, 2 z_1 z_2)
	\]
	is separable in $(z_1, z_2)$ coordinates. When the function $V(q)$ satisfies this separability condition, we call such systems (\ref{eq: original_H}) \emph{Stark-type systems}. By using the separability of Stark-type systems, we obtain infinitely many integrable Stark-type billiard systems as we state in the following Theorem
	\begin{theorem}
		\label{thm: sterk-type}
		There exists infinitely many potential functions $V$ such that the system
		\[
		H = \frac{|p|^2}{2} {-} \frac{s}{|q|} - V(q_1, q_2)
		\]
		allows any focused parabola with the $q_1-$axis as the main axis as an integrable reflection wall.
	\end{theorem}
	\begin{proof}
		Assume that the system $H$ is {of} Stark-type, so that the transformed Hamiltonian is separable, 
		\emph{i.e.}
		\[
		\hat{H} = \frac{|w|^2}{2} + s {-} f (z_1^2 + z_2^2)+ (z_1^2 + z_2^2)V(z_1^2 - z_2^2, 2 z_1 z_2) = \hat{H}_1(z_1,w_1) + \hat{H}_2(z_2,w_2).
		\]
		 From its separability, this system has the additional first integral $\hat{H}_1(z_1,w_1)${, which is invariant under the reflections against a line which is parallel to the $z_1-$ or $z_2-$axis. Now since any lines which is parallel to the $z_1-$ or $z_2-$axis is transformed into a focused parabola in the form of}
		 \[
		 q_1 = -\frac{q_2^2}{4c^2} + c^2
		 \]
		 or 
		 \[
		 q_1 = \frac{q_2^2}{4c^2} - c^2
		 \]
		 by the mapping $z \mapsto z^2$, {by} Theorem \ref{thm: conformal_trans}, the original system allows any focused parabola with the $q_1-$axis as the main axis as an integrable reflection wall.
		  
		We are just left to show that there exists infinitely many Stark-type systems. {We assume that the function $V$ depending only on $z_1^2 - z_2^2$ and $2 z_1 z_2$ satisfies 
		\[
		(z_1^2 + z_2^2)V(z_1^2 - z_2^2, 2 z_1 z_2) = g_1(z_1) + g_2(z_2)
		\]
		for some smooth even functions $g_1, g_2 \in C^{\infty}(\mathbb{R}, \mathbb{R})$, i.e. $g_1 (-z_1)= g_1(z_1)$ and $g_2(-z_2) = g_2(z_2)$ for all $z_1, z_2 \in \mathbb{R}$.  We then define $V(z_1^2 - z_2^2, 2 z_1 z_2):= \frac{g_1(z_1) + g_2(z_2)}{z_1^2 + z_2^2}$, and we may then solve $V$ as a function of $q_{1}=z_1^2 - z_2^2$ and $q_{2}=2 z_1 z_2$, which is possible since $\frac{g_1(z_1) + g_2(z_2)}{z_1^2 + z_2^2}$ is centrally symmetric. }
	\end{proof}
	{This theorem is an analogue of \cite[Theorem 3.1]{Cieliebak-Frauenfelder-Koert} in the setting of mechanical billiards.}

	\subsection{Examples of Stark-type Billiard Systems}
	In the following, we {discuss some} concrete examples of Stark-type systems.

	\paragraph{Stark problem}
	Firstly we consider the Stark problem by setting $V(q)= g q_1$. The Stark problem can be interpreted as a planer system consists of gravitational potential and an external constant force field. The Hamiltonian of this problem is given by 
	\[
	H = \frac{|p|^2}{2} {-} g q_1 - \frac{s}{|q|}.
	\]
	which on its energy hypersurface $\{H + f = 0\}$ is then transformed into the system
	\[
	\hat{H}=\frac{|w|^2}{2} - g(z_1^4 - z_2^4) + f(z_1^2 + z_2^2) {-} s.
	\]
	which is separable in $(z_1,z_2)$ coordinates. {From this we get}
	\begin{cor}
		\label{cor: stark}
		The Stark problem  {$(\mathbb{R}^{2} \setminus O, g_{flat}, \frac{s}{|q|}+g q_{1}  )$}
		admit any focused parabola with the $q_1-$axis as the main axis as integrable reflection wall. {In particular, by setting respectively $s=0$ we get that any focused parabola with the $q_1-$axis as the main axis is an integrable reflection wall in a uniform gravitational field along the $q_{1}$-direction.}
	\end{cor}
	Note that this argument on the integrability of the Stark problem using conformal transformation {provides} an alternative proof of the theorem of Korsch-Lang \cite{Korsch-Lang}.
	\paragraph{Frozen-Hill's Problem with  Centrifugal Correction}
	{Setting {$V= g q_1^2 + g q_2^2/4 $} gives rise to the \emph{so-called} frozen-Hill's problem with centrifugal corrections \cite{Cieliebak-Frauenfelder-Koert}.}
	The Hamiltonian of this system is given by 
	\[
	H = \frac{|p|^2}{2}  {-} \frac{s}{|q|} - g q_1^2 - \frac{g}{4} q_2^2.
	\]
	{Similarly as in the case of Stark problem, on its energy-hypersurface $\{H + f = 0\}$ we transform the system into}
	\[
	\hat{H}= \frac{|p|^2}{2}  {-} s  - (z_1^2 + z_2^2)(g(z_1^2 - z_2^2)^2 + g z_1^2 z_2^2 - f),
	\]
	{which can be written as}
	\[
	\hat{H}= \frac{|p|^2}{2}  {-} s - g(z_1^6 + z_2^6) + f(z_1^2 + z_2^2)
	\]
	which is separable in $(z_1,z_2)$ coordinates. {We thus get}
		\begin{cor}
		\label{cor: frozen-hill}
		The frozen Hill's problem with centrifugal corrections  {$(\mathbb{R}^{2} \setminus O, g_{flat}, \frac{s}{|q|}+g q_1^2 + \frac{g}{4} q_2^2)$}
		admits any focused parabola with the $q_1-$axis as the main axis as integrable reflection wall.
	\end{cor} 

	\section{Integrable mechanical billiards of two-center problem}
	\label{sec: two_center_problem}
	We now consider the two center problem in the plane $\mathbb{C}$ with the two centers at $-1,1 \in \mathbb{C}$. The Hamiltonian of this system with mass factors {$m_{1}, m_{2}$ (which can take both signs),} is given by 
	\[
	{H=\frac{|p|^2}{2} - \frac{m_{1}}{|q-1|}  - \frac{m_{2}}{|q+1|}.}
	\]
        A classical way to show the integrability of this system uses its separability in elliptic-hyperbolic coordinates \cite{Varvoglis}. 
	Set $r_1= |q-1|, r_2=|q+1|$ and define the elliptic-hyperbolic coordinates as
	\[
	\xi = \frac{r_1 + r_2 }{2}, ~ \eta = \frac{r_1 - r_2 }{2}.
	\] 
	{In this coordinate system,  the curves $\xi = \hbox{const.}$ and $\eta = \hbox{const.}$ describe, respectively, ellipses and  branch of hyperbolae in the plane. Note that confocal ellipses in general intersect a branch of {confocal} hyperbola in $0$ or $2$ points and thus the change of coordinates $q \mapsto (\xi, \eta)$ is in general a 2-to-1 transformation.}
         The above Hamiltonian is then transformed into 
	\[
	{H( p_{\xi}, p_{\eta}, \xi, \eta) = \frac{1}{\xi^2 - \eta^2} \left( \frac{1}{2} (\xi^2 -1)^2 p_{\xi}^2 -{(m_1 + m_2)} \xi +  
	\frac{1}{2} (\eta^2 -1)^2 p_{\eta}^2 + { (m_{1}-m_{2})} \eta  \right).} 
	\] 
	in which $(p_{\xi}, p_{\eta})$ are the conjugate coordinate to $(\xi, \eta)$ respectively.
	By fixing $H = -f$ and changing the time by multiplying the Hamiltonian $H + f$ by $(\xi^2 - \eta^2)$ on the zero energy surface, we obtain the new Hamiltonian
	\[
	{K( p_{\xi}, p_{\eta}, \xi, \eta) = \frac{1}{2} (\xi^2 -1)^2 p_{\xi}^2 -{(m_1 + m_2)} \xi + \frac{1}{2} (\eta^2 -1)^2 p_{\eta}^2 +{ (m_{1}-m_{2})} \eta + f(\xi^2 - \eta^2)}
	\]
	{which is separable, showing its integrability.}
	{The curves $\xi = \hbox{const.}$ and $\eta = \hbox{const.}$ actually give integrable reflection walls of the two-center problem, as we shall establish below. Note that the elliptic-hyperbolic coordinate system is not conformal, therefore we have to use the following approach. }
	
	{{The conformal mapping that we are going to use for our purpose is the following one by Birkhoff \cite{Birkhoff_1915}}:
	\[
	z \mapsto q = \frac{z + z^{-1}}{2}, \mathbb{C}\backslash\{ 0\} \to \mathbb{C},
	\]
	{in real coordinates we have }
	\[
	q_1 = z_1 + \frac{z_1}{z_1^2 + z_2^2}, \quad
	q_2 =z_2 -  \frac{z_2}{z_1^2 + z_2^2}
	\]
	{which is conjugate to the complex square mapping by a M\"obius Transformation \cite{Waldvogel}, \cite{Cieliebak-Frauenfelder-Zhao}. }
	
	We use the cotangent lift of this mapping, given by the expression
	\[
	q = \frac{z + z^{-1}}{2} ,\quad p = \frac{2w}{1 - \bar{z}^{-2}}.
	\] 
	 to pull the shifted Hamiltonian $K = H-f$ back {to the expression}
	 \[
	 \frac{2|z|^4 |w|^2}{|z+1|^2 |z-1|^2} - \frac{2 m_1 |z|}{|z-1|^2} - \frac{2 m_2 |z|}{|z+1|^2} +f.
	 \]
	{By changing time on the zero-energy hypersurface, we obtain the new Hamiltonian}
	\[
	\hat{K}=\frac{|w|^2}{2} - \frac{m_1 |z+1|^2}{2|z|^3} - \frac{m_2 |z-1|^2}{2|z|^3} + f \frac{|z-1|^2|z+1|^2}{4|z|^4}=0.
	\]
	}
	\begin{prop}\label{prop: transformation 2CP}
		The mapping $z \mapsto  \frac{z + z^{-1}}{2}$ {pulls} confocal ellipses back to two centered circles, and {pulls} confocal hyperbolae to {a pair of lines passing through the center.}
	\end{prop}
	\begin{proof}
		A confocal ellipse is given by the equation
		\begin{equation} \label{equation: confocal ellipse}
		\frac{q_1^2}{b^2 +1} + \frac{q_2^2}{b^2} -1 = 0.
		\end{equation}
		{with $b >0$ as a parameter. }
		
		With the conformal mapping {we use, the LHS of the above equation} is transformed into 
		\begin{align*}
		&\frac{(b^2 z_1^2 +  (b^2 + 1)z_2^2)((z_1^2 + z_2^2)^2 + 1) + (2 b^2 z_1^2 -2(b^2 + 1)z_2^2) (z_1^2 + z_2^2) }{4b^2(b^2 + 1)(z_1^2 + z_2^2)^2}-1\\
		=& \frac{(b^2 z_1^2 +  (b^2 + 1)z_2^2)((z_1^2 + z_2^2)^2 - 2(2b^2 + 1)(z_1^2 + z_2 ^ 2)+ 1)}{4b^2(b^2 + 1)(z_1^2 + z_2^2)^2}.
		\end{align*}
		and thus the transformed equation is equivalent to
		$$(z_1^2 + z_2^2)^2 - 2(2b^2 + 1)(z_1^2 + z_2 ^ 2)+ 1=0$$
		{which, seen as a quadratic equation of $z_1^2 + z_2 ^ 2$, has two positive solutions, giving rise to two centered circles.}
		
		For confocal hyperbolae, we set $b$ {in \eqref{equation: confocal ellipse}} as a purely imaginary number such that $b^2 + 1 >0$, then the equation
		\[
		(z_1^2 + z_2^2)^2 - 2(2b^2 + 1)(z_1^2 + z_2 ^ 2)+ 1=0
		\]
		{has no real-valued solutions and we get that the transformed equation of \eqref{equation: confocal ellipse} is equivalent to }
		$$(b^2 z_1^2 +  (b^2 + 1)z_2^2)=0$$ 
		{which describes a pair of lines passing through the origin.}
		Note that they are {the} two asymptotes of the confocal hyperbola
		\[
		\frac{{z_1}^2}{b^2 +1} + \frac{{z_2}^2}{b^2} -1 = 0.
		\]
	\end{proof}
	{The separability of the (properly-transformed) two-center Hamiltonian in the  elliptic-hyperbolic coordinates is thus equivalent to the separability of $\hat{K}$ in polar coordinates. We now verify the latter.}

	{We set $z = r e^{ i \theta}$, and denote the conjugate momenta  by $p_{\theta}, p_r$ respectively.}
	{ Explicitly we have $w =p_{r} \bm{e}_{r} + \dfrac{p_{\theta}}{r} \bm{e}_{\theta} $.}
	The transformed Hamiltonian ${\hat{K}}$ into the polar coordinates $(p_r, p_\theta, r, \theta)$ with zero energy zero becomes
	\[
	\hat{K}=\frac{1}{2} \left( p_r^2 + \frac{p_{\theta}^2}{r^2} -\frac{2(m_1 - m_2)\cos \theta}{r^2} -\frac{(m_1 + m_2 )(r^2 + 1)}{r^3} + 2f\frac{r^4 + r^2 + 1 -4 r^2 \cos^2 \theta } {r^4} \right)=0.
	\]
	By multiplying this by $2r^2$, we obtain
	\[ 
	r^2 p_r^2 + p_{\theta}^2  -2(m_1 - m_2)\cos \theta -\frac{(m_1 + m_2 )(r^2 + 1)}{r} + 2f\frac{r^4 + r^2 + 1} {r^2} - 8 f \cos^2 \theta =0,
	\]
	{which is now seen to be separable. From this we have the following additional first integral}
	\[
	r^2 p_r^2  -\frac{(m_1 + m_2 )(r^2 + 1)}{r}+ 2f\frac{r^4 + r^2 + 1} {r^2},
	\]
	showing the integrability of the system.
	
	In{ the} next lemma, we {establish} the integrability of centered circular reflection walls and centered line reflection walls {in} this
	system.
	
	\begin{lemma}
		\label{lem:2CP}
		{Any combination of centered circles and lines passing through the origin are integrable reflection walls for the system $\hat{K}$ (at its zero-energy level).}
	\end{lemma}
	\begin{proof}
		It is sufficient to check invariance of $p_r^2$ before and after the reflection against the reflection walls. For centered circular reflection walls, {the} $\theta$-component $r \dot{\theta}$ of the conjugate momenta $w$ is preserved and the sign of {the} $r$-component $\dot{r}$ is switched after the reflection. {At a line passing through the origin,} {the} $r$-component is preserved and the sign of {the} $\theta$-component is switched after the reflection. Hence, in both cases, the value $p_r^2= \dot{r}^2$ is unchanged before and {after the reflection agains these reflection walls.}
	\end{proof}

		We now deduce the following theorem for billiards defined with the two-center problems: 
		\begin{theorem}
			The two center problem in the plane admits any {combination of} confocal ellipses and confocal hyperbolae as an integrable reflection wall. 
		\end{theorem}
		\begin{proof}
			{This follows directly from Proposition \ref{prop: transformation 2CP} and Lemma \ref{lem:2CP}.}
		\end{proof}
	
		\begin{rem}
			 {By letting one mass parameter in the two center problem be zero, we obtain the the Kepler billiards with any combination of confocal focused ellipses/hyperbolae as an integrable reflection wall directly form the theorem above. This thus provides an alternative proof of Corollary \ref{cor: confocal_Kepler}.}
		\end{rem}
	 
	 \begin{rem}
	 		{By letting one mass parameter be zero and sending it to infinity, we obtain the integrable Kepler billiards with a focused parabola as the limiting cases from focused ellipses or hyperbolae. Additionally we may also deduce the same result for focused parabolae with collinear major axes as the limiting case from combinations of focused ellipses/hyperbolae. This argument provides an alternative proof of Corollary \ref{cor: confocal_parabola_halfline_Kepler}.}
	 \end{rem}
		
	{\bf Acknowledgement}	
	We benefit from discussions with Alain Albouy and Gert Heckman. A.T. is supported by Masason Foundation. L.Z. is supported by DFG ZH 605/1-1. 
	
	\newpage
	
	
	
	\appendix
	\section{Integrability of conic section boundaries of free planar billiards}
	\label{sec: Joachimsthal}
		Here we will give a proof for integrability of conic section boundaries in free motion case.

	 Consider the elliptic/hyperbolic reflection walls in the form
		\[
		\frac{x_1^2}{a^2} \pm \frac{x_2^2}{b^2} = 1.
		\]
		The classical Joachimsthal integral can be written in the form of the product of the velocity and normal vector as follows:
		\[
		J(x,v) := -\frac{1}{2} \langle v, \nabla f(x)  \rangle,
		\]
		where $f = x_1^2/a^2 \pm x_2^2/b^2$ and $x$ lies in $f = 1$.
		Let $(x,v)$ be the pair of reflection point and the reflected vector at $x$, and let $(x',v')$ be the consecutive reflection point and the reflected vector at $x'$. Then we will check that 
		\[
		J(v,x) - J(v',x')= -\frac{1}{2} \langle v, \nabla f(x)  \rangle +  \frac{1}{2} \langle v', \nabla f(x')  \rangle = 0.
		\]
		Since the reflection is elastic, the vector $v + v'$ is tangent to the ellipse/hyperbola at $x'$ and $\nabla f(x')$ is normal to the ellipse/hyperbola at $x'$, hence we have
		\[
		\langle v +v', \nabla f(x')  \rangle = 0.
		\]
		Using this to substitute $v'$, we only need to show that
		\[
		\langle v, \nabla f(x) +  \nabla f(x')  \rangle = 0.
		\]
		Additionally, we know that $v$ and $x -x'$ agree up to some scaling, hence it suffices to show that
		\[
		\langle x - x', \nabla f(x) +  \nabla f(x')  \rangle = 0.
		\]
		Now we write 
		\[
		\langle x - x', \nabla f(x) +  \nabla f(x')  \rangle = \langle x, \nabla f(x)  \rangle +  \langle x, \nabla f(x')  \rangle -  \langle x', \nabla f(x)  \rangle - \langle x', \nabla f(x')  \rangle.
		\]
		Notice that $x$ and $x'$ are points of the ellipse/hyperbola $f= c$, therefore we have $\langle x, \nabla f(x)  \rangle =  \langle x', \nabla f(x')  \rangle = 2$. Also, we get $ \langle x, \nabla f(x')  \rangle -  \langle x', \nabla f(x)  \rangle = 0$ from the direct computation. As the conclusion, $J$ is preserved under the reflection at the elliptic/hyperbolic reflection wall. Note that in Proposition \ref{prop:free_conic} we give an alternative proof for the integrability of elliptic/hyperbolic reflection walls.
	
		Next, we consider parabolic reflection walls. In this case, the additional first integral is given by 
		\[
		\gamma = C \cdot \sin \theta,
		\]
		where $C$ is the angular momentum with respect to the focus of the parabola, and $\theta$ is the angle that the incoming vector makes in a counter-clockwise direction with the axis of symmetry {of} the parabola. We here employ the part of the proof for the integrability of confocal parabolae boundaries appears in \cite{Poet}. There are three cases to consider; (1) the incoming vector cuts the segment between the apex and the focus of the parabola, (2) goes through the outside of the focus, (3) passes the focus, or goes parallel to the axis of symmetry. 
		We here describe the proof for the second case. Figure \ref{fig:parabola_consv} illustrates this case (2); the incoming line segment $IB$ goes through the outside of the focus and gets reflected back at $B$. The outgoing direction is given by $BR$. Denote the focus of the parabola by $F$ and set the perpendicular line from $F$ to the line $IB$ and denote the intersection point by $K$. Likewise, we denote the intersection point of the line $BR$ and the perpendicular line from $F$ to $IB$, by $L$. Construct the line $BG$ which is parallel to the axis of symmetry. Additionally, let $BN$ {be} normal to the parabola at $B$. Set {$C$} and {$C'$} be the angular {momenta} with respect to the focus of the incoming and outgoing vectors, respectively. Then the quantities before and after reflection $\gamma, \gamma'$ are given by
		\[
		\gamma = C \cdot \sin \theta, \quad \gamma' = C' \cdot \sin \theta',
		\]
		where $\theta, \theta'$ are the {angles made by} $IB$ and $BR$ from the axis of symmetry, respectively. We will show that $\gamma = \gamma'$. {For this to hold, it is enough to show this while replacing the angular momenta $C$ and $C'$ respectively by $|FK|$ and $|FL|$ in the expression}. {Set $\angle FBN = \angle NBG = \alpha$ and $\angle IBF = \angle RBG= \beta$ in which the angles are non-oriented}. Then we have
		\[
		\sin \theta = \sin (2 \alpha - \beta), \quad \sin \theta' = \sin \beta,
		\] 
		and
		\[
		|FK| = |FB|\sin \beta,\quad |FL| = |FB| \sin (2 \alpha - \beta).
		\]
		Thus, we get 
		\[
		\gamma = |FB|\sin \beta \sin (2 \alpha - \beta) = \gamma'.
		\]
		The proof for the case $(1)$ proceeds in a similar way and its details are given in \cite{Poet}. The proof for the case (3) immediately follows from the fact that the parallel line to the axis is reflected directly to the focus and vice versa.

	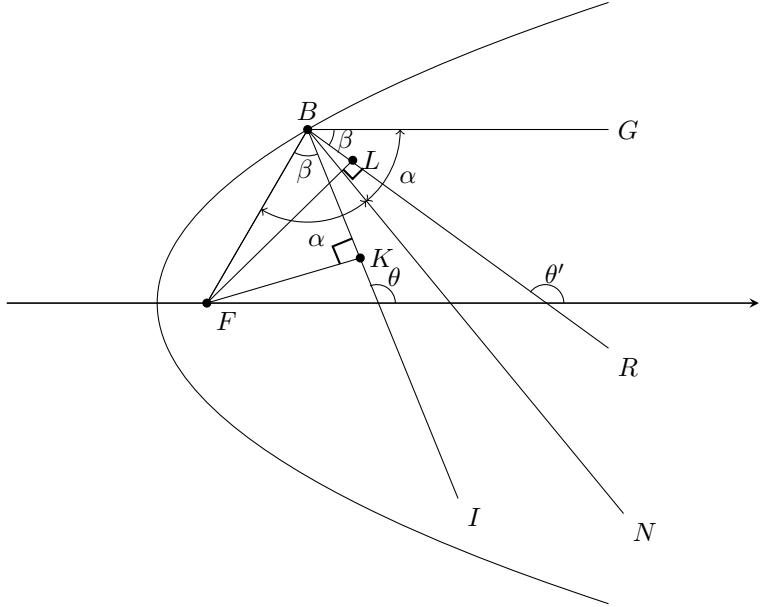
\begin{figure}
		\begin{tikzpicture}[scale=2.0]
		\draw[->,>=stealth,semithick] (-1,0)--(4,0); 
		\coordinate (F) at (0.33,0);
		\draw (F) node[below right]{$F$}; 
		\fill[black] (F) circle (0.03);
		\draw[rotate=-90] (-2,3) parabola bend (0,0) (2,3);
		\coordinate (B) at (1,1.1547);
		\fill[black] (B) circle (0.03);
		\draw (B) node[above]{$B$}; 
		\coordinate (I) at (2,-1.3);
		\draw (I) node[below right]{$I$}; 
		\coordinate (R) at (3,-0.3);
		\draw (R) node[below right]{$R$};
		\draw (I)--(B);
		\draw (B)--(R);
		\coordinate (N) at (3.1,-1.4);
		\draw (N) node[below right]{$N$};
		\draw (B)--(N);
		\coordinate (G) at (3,1.1547);
		\draw (G) node[right]{$G$};
		\draw (B)--(G);
		\draw (B)--(F);
		\coordinate (K) at (1.35,0.3);
		\fill[black] (K) circle (0.03);
		\draw (K) node[right]{$K$};
		\draw (K)--(F)--(B);
		\draw[thick] ($(K)!4pt!(B)$)--($(K)!4pt!(B)!4pt!90:(B)$)--($(K)!4pt!(F)$);
		\coordinate (L) at (1.3,0.95);
		\fill[black] (L) circle (0.03);
		\draw (L) node[right]{$L$};
		\draw (L)--(F);
		\draw[thick] ($(L)!2.5pt!(F)$)--($(L)!2.5pt!(F)!2.5pt!90:(F)$)--($(L)!2.5pt!(R)$);
		\draw pic["$\beta$", draw, angle eccentricity=1.6, angle radius=10] {angle=F--B--I};
		\draw pic["$\alpha$", draw,<->, angle eccentricity=1.2, angle radius=35] {angle=F--B--N};
		\draw pic["$\beta$", draw, angle eccentricity=1.5, angle radius=10] {angle=R--B--G};
		\draw pic["$\alpha$", draw,<->, angle eccentricity=1.2, angle radius=35] {angle=N--B--G};
		\coordinate (Z) at (20,0);
		\coordinate (X) at (1.46,0);
		\coordinate (Y) at (2.58,0);
		\draw pic["$\theta$", draw, angle eccentricity=1.7, angle radius=7] {angle=Z--X--B};
		\draw pic["$\theta'$", draw, angle eccentricity=1.7, angle radius=7] {angle=Z--Y--B};
		\end{tikzpicture}
		\caption{illustration of parabolic boundary in case (2) }
		\label{fig:parabola_consv}
	\end{figure}

	For parallel two lines reflection walls, it is trivial that the reflection angle in preserved.
	
	As a conclusion, any conic section including degenerate ones are integrable reflection walls for free billiards.

	\section{Invariance of Transformed Jaochimsthal First Integral}
	\label{sec: transformed_Jaochimsthal}
	{
		{In this Appendix, we consider a special case of Theorem \ref{thm: duality} with $k=2$ and $s=0$, and we verify the integrability of the mechanical billiard system thus obtained on its zero-energy level with direct computation. }
		
		{When $k=2$ and $s=0$, the conformal mapping $z \mapsto z^2$ gives a transformation between the free motion 
		\[
		H= \frac{|p|^2}{2} =  f 
		\]
		on its $f$-energy level, {$f>0$} and the repulsive Hooke system
		\[
		\hat{H}=\frac{|w|^2}{2} - f |z|^2 = 0
		\]
		on its zero-energy level. }
		
		We take a non-centered ellipse
		\[
		\frac{(q_1-c_1)^2}{a^2} + \frac{(q_2-c_2)^2}{b^2}=1
		\]
		{which is an} integrable reflection wall for free billiard. As one can see in Appendix \ref{sec: Joachimsthal}, the Joachimsthal first integral is given by 
		\[
		\frac{\tilde{q_1} p_1}{a^2} + \frac{\tilde{q_2} p_2}{b^2},
		\]
		where {$(\tilde{q_1}, \tilde{q_2})$} is the {point of reflection}. {For our purpose, to simplify the computations, we consider the squared Joachimsthal first integral which we interpolate along the free flow as}:
		\[
		J:= \frac{(b + c_2 -q_2)(b - c_2 + q_2) p_1^2 + 2 p_2 (-q_2 + c_2)(- q_1 + c_1)p_1 + p_2^2 (a + c_1 - q_1)(a -c_1  + q_2)}{a^2 b^2}.
		\]
		By the mapping $z \mapsto z^2$, the non-centered elliptic reflection wall is transformed into 
		\[
		\frac{(z_1^2 - z_2^2 -c_1)^2}{a^2} + \frac{(2z_1 z_2 -c_2)^2}{b^2}=1.
		\]
		{We now transform the first integral $J$ by the same mapping. With Maple, we obtained the following form:}
		\begin{align*}
		&\hat{J}= \frac{1}{(z_1^2 + z_2^2)a^2 b^2}\cdot (-w_2^2 z_1^6 + 2 w_1 w_2 z_1^5 z_2 + ((-w_1^2 - 2 w_2^2)z_2^2 + 2 c_1 w_2^2 - 2 c_2 w_1 w2)z_1^4 \\
		&+ 2(2 w_2 w_1 z_2^2 + c_2 (w_1^2 + w_2^2)) z_2 z_1^3 + ((-2 w_1^2 - w_2^2) z_2^4 + (-2 c_1 w_1^2 + 2 c_1 w_2^2 - 4 c_2 w_1 w_2) z_2^2 \\
		&+ w_1^2 (b^2 - c_2^2) + 2 c_1 c_2 w_1 w_2 + w_2^2 (a^2 - c_1^2)) z_1^2 + 2 (w_1 w_2 z_2^4 + c_2 (w_1^2 + w_2^2) z_2^2\\
		& + c_1 c_2 w_1^2 + w_1(a^2 - b^2 - c_1^2 + c_2^2) w_2 - c_1 c_2 w_2^2) z_2 z_1 + (-w_1^2 z_2^4 + (-2 c_1 w_1^2 - 2 c_2 w_1 w_2)z_2^2  \\
		&+ (a^2 - c_1^2) w_1^2 - 2 c_1 c_2 w_1 w_2 + w_2^2(b^2 - c_2^2)) z_2^2).
		\end{align*}
		{A} direct {(but unnecessary)} computation {with Maple} shows that 
		\[
		\{\hat{H}, \hat{J}\}= \sum_{i =1,2} \frac{\partial \hat{H}}{\partial w_i} \frac{\partial \hat{J}}{\partial z_i} - \frac{\partial \hat{H}}{\partial z_i} \frac{\partial \hat{J}}{\partial w_i}=0
		\]
		on {$\{\hat{H}=0\}$}. This means that $\hat{J}$ is invariant along the {transformed} flow on {$\{\hat{H}=0\}$}.
		
		Now we {verify the} invariance of $\hat{J}$ before and after the reflection against the transformed reflection wall.
		Set 
		\[
		F: = \frac{(z_1^2 - z_2^2 -c_1)^2}{a^2} + \frac{(2z_1 z_2 -c_2)^2}{b^2}-1
		\]
		and define $F_1 := \partial F/ \partial z_1$ and $F_2 := \partial F / \partial z_2$.
		{The normal vector to the curve $\{F=0\}$ is} given by 
		\[
		n:= {( F_1, F_2 )}
		\]
		{and thus the} normal component of $w:= (w_1,w_2)$ is obtained as
		\[
		{w_{n}}:= \frac{w \cdot n}{|n|^2} n= \left( \frac{(w_1 F_1 + w_2 F_2 )F_1}{F_1^2 + F_2^2},  \frac{(w_1 F_1 + w_2 F_2 )F_2}{F_1^2 + F_2^2}   \right).
		\]
		From the law of elastic reflection, the momenta after the reflection $w' := (w'_1, w'_2)$ is described as 
		\[
		w' = w - 2 {w_{n}}.
		\]
		{The difference before and after the reflection is computed as}
		\[
		\hat{J}(z_1, z_2, w_1', w_2') - \hat{J}(z_1, z_2, w_1, w_2) =  D_1 {\cdot} D_2,
		\]
		where
		\[
		D_1:= (- z_1^2 + z_2^2 + z + c_1)(z_1^2 -z_2^2 + a - c_1)b^2 - (-2z_1 z_2+ c_2)^2 a^2,
		\]
		and $D_2$ is a polynomial of $z_1, z_2,w_1$, and $w_2$.
		Since $ F \cdot a^2 \cdot b^2 = -D_1$, the factor $D_1$ becomes $0$ {at} the reflection wall $\{F = 0\}$. Therefore, $\hat{J}$ is invariant under the reflection. This means the transformed first integral $\hat{J}$ is the first integral for the billiard system $\hat{H}=0$ with the transformed reflection wall {$\{F = 0\}$} on the zero-energy surface. 
	}
	
	\section{Invariance of Gallavotti-Jauslin's First Integral}
	\label{sec: G-J}
	{Here, we directly verify the invariance of Gallavotti-Jauslin's first integral which appeared in Lemma \ref{lem: Gallavotti-Jauslin} {of} the Kepler billiard with $s=1$ with a focused elliptic and a focused hyperbolic reflection wall.} 
	By ruling out the rotational symmetry, we can write a focused ellipse as
	\[
	\frac{(q_1 - \sqrt{a^2 -b^2})^2}{a^2} + \frac{q_2^2}{b^2} = 1.
	\]
	Set $F := \frac{(q_1 - \sqrt{a^2 -b^2})^2}{a^2} + \frac{q_2^2}{b^2}-1$ and define $F_1 := \partial F / \partial q_1$ and $F_2 := \partial F / \partial q_2$.
	Let $(p_1,p_2)$ and $(p_1',p_2')$ denote momenta, respectively, before and after the reflection against this focused conic section reflection wall. From the law of elastic reflection, we obtain
	\[
	p_1' = p_1 - \frac{2(p_1 F_1 + p_2F_2)F_1}{F_1^2 + F_2^2},
	\]
	\[
	p_2' =p_2 - \frac{2(p_1 F_1 + p_2F_2)F_2}{F_1^2 + F_2^2}.
	\]
	Now we {test the invariance of a first integral of the form}
	\begin{align*}
	A := (-p_1 q_2 + p_2 q_1)^2 + l_1\left((-p_1 q_2 + p_2 q_1)p1 + \frac{q2}{\sqrt{q_1^2 + q_2^2}}\right) \\
	+ l_2\left((-p_1 q_2 + p_2 q_1)p2 - \frac{q1}{\sqrt{q_1^2 + q_2^2}}\right)
	\end{align*}
	under the reflection against the reflection wall. The difference between the value of $A$ before and after the reflection is {computed as }
	\begin{align*}
	&A(q_1,q_2,p_1,p_2) - A(q_1,q_2,p_1',p_2')= \frac{-(F_1 p_2 - F_2 p_1)(F_1 p_1 + F_2 p_2)}{(F_1^2 + F_2^2)^2}\times\\
	&(F_1^2((l_1 - 2 q_2)q_1 - l_2 q_2) + 2 F_1 F_2(q_1^2 + l_2 q_1 + q_2(l_1 - q_2))-F_2^2 ((l_1 - 2 q_2 )q_1 - l_2 q_2)).
	\end{align*}
	Set
	\[
	G:= F_1^2((l_1 - 2 q_2)q_1 - l_2 q_2) + 2 F_1 F_2(q_1^2 + l_2 q_1 + q_2(l_1 - q_2))-F_2^2 ((l_1 - 2 q_2 )q_1 - l_2 q_2).
	\]
	When $l_1 = 0, l_2 = -2 \sqrt{a^2 - b^2 }$, $G$ becomes
	\[
	G = \frac{8 q_2(q_1 - \sqrt{a^2 - b^2})(a - b)(a + b)(q_2^2 a^2 - b^4 + b^2 q_1^2 -2 \sqrt{a^2 - b^2} b^2 q_1)}{a^4 b^4}
	\]
	{which is 0 at the reflection wall $\{F=0\}$. Note that $l_{2}=-2\tilde{a}$ as appeared in Lemma \ref{lem: Gallavotti-Jauslin} which in this case represents the center-focus distance of the ellipse under concern. }
	
	Analogously, we also get the integrability of focused hyperbolae reflection wall by setting $b$ as a purely imaginary number.
	
\end{quote}




\end{document}